\documentclass[11pt,reqno,final]{article}

\usepackage{amsmath,amsfonts,amssymb,amsthm,version}
\usepackage{mathrsfs,fancybox,pifont}
\usepackage{graphicx}
\usepackage{url,hyperref}
\usepackage[notcite,notref]{showkeys}
\usepackage{color}
\usepackage{subfigure,multirow}
\usepackage{epstopdf}
\usepackage{cases}
\usepackage{mathtools}
\usepackage{algorithm,algorithmic}
\usepackage{authblk}
\usepackage{lipsum}

\numberwithin{equation}{section} %��ʽ�����źţ�����ȥ��

\theoremstyle{plain}
\newtheorem{exam}{Example}[section]
\newtheorem{theorem}[exam]{Theorem}
\newtheorem{lemma}[exam]{Lemma}
\newtheorem{remark}[exam]{Remark}

\newtheorem{proposition}[exam]{Proposition}
\newtheorem{definition}[exam]{Definition}

\begin{document}
\date{}

\title{ Linearization and H\"{o}lder continuity of generalized ODEs with application to measure differential equations
\footnote{  This paper was jointly supported by the National Natural Science Foundation of China (No. 11671176, 11931016). %and Natural Science Foundation of Zhejiang Province (No.LZ23A010001).
}
}
\author
{
Weijie Lu$^{a}$\,\,\,\,\,
Yonghui Xia$^{a}\footnote{Corresponding author. yhxia@zjnu.cn. }$
\\
{\small \textit{$^a$ College of Mathematics Science,  Zhejiang Normal University, 321004, Jinhua, China}}\\
{\small Email: luwj@zjnu.edu.cn; yhxia@zjnu.cn}
}

\maketitle

\begin{abstract}
In this paper, we study the topological conjugacy between the linear generalized ODEs (for short, GODEs)
\[
  \frac{dx}{d\tau}=D[A(t)x]
\]
and their nonlinear perturbation
\[
  \frac{dx}{d\tau}=D[A(t)x+F(x,t)]
\]
on Banach space $\mathscr{X}$, where $A:\mathbb{R}\to\mathscr{B}(\mathscr{X})$ is a bounded linear operator on $\mathscr{X}$ and $F:\mathscr{X}\times \mathbb{R}\to \mathscr{X}$ is Kurzweil integrable.
GODEs are completely different from the classical ODEs.
Note that the GODEs in Banach space are defined via its solution. $\frac{dx}{d\tau}$ is only a notation and  %$\frac{dx}{d\tau}$
it does not indicate that the solution has a derivative.  The solution of the GODEs can be discontinuous and even the number of discontinuous points is countable,
so that many classical theorems and tools are no longer applicable to the GODEs.
  For instances, the chain rule and the multiplication rule of derivatives, differential mean value theorem, and integral mean value theorem are not valid for the GODEs.
In this paper, we study the linearization and its H\"{o}lder continuity of the GODEs.
   Firstly, we construct the formula for bounded solutions of the nonlinear GODEs in the Kurzweil integral sense.
   Afterwards, we establish a Hartman-Grobman type linearization theorem which is a bridge connecting the linear GODEs with their nonlinear perturbations. Further, we show that the conjugacies are both H\"{o}lder continuous by using the Gronwall-type inequality (in the Perron-Stieltjes integral sense) and other nontrivial
estimate techniques. %The GODEs include measure differential equations,  impulsive differential equations, functional differential equations and the classical ordinary differential equations as special cases.
  Finally, applications to the measure differential equations and impulsive differential equations, our results are very effective.
\end{abstract}

%26A39 Denjoy and Perron integrals, other special integrals
% 34A36 Discontinuous ordinary differential equations
%37C15 Topological and differentiable equivalence, conjugacy, moduli, classification of dynamical systems
%45N05 Abstract integral equations, integral equations in abstract spaces

\section{Introduction}
  \subsection{History of the generalized ODEs}

  The generalized ordinary differential equations (for short, GODEs) in Banach space have received widespread attentions recently, which has the form of:
\begin{equation}\label{1F}
  \frac{dx}{d\tau}=DK(x,t),
\end{equation}
where $K:\mathscr{X}\times \mathbb{R}\to \mathscr{X}$ is a given map on Banach space $\mathscr{X}$.
    In addition, as mentioned by Schwabik (\cite{S-BOOK1992}, Remark 3.2),
the letter $D$ of \eqref{1F} means that \eqref{1F} is a generalized differential equation, this concept being defined via its solution,
and  $\frac{dx}{d\tau}$ is only a notation. Even, the symbol $\frac{dx}{d\tau}$ does not indicate that the solution has a derivative.
To illustrate this result,  recall an example from Schwabik's book (see \cite{S-BOOK1992}, pp. 100):
  let $g:[0,1]\to \mathbb{R}$ be a continuous function whose derivative does not exist at any point in $[0,1]$.
  We  define $K(x,t)=g(t)$. In this case, from the definition of the integral, it follows that
\[
\int_c^d DK(x(\tau),\gamma)=\int_c^d Dg(\gamma)=g(d)-g(c), \quad \mathrm{for} \; c,d \in [0,1],
\]
which implies that $x:[0,1]\to\mathbb{R}$ given by $x(\gamma)=g(\gamma)$ for $\gamma\in [0,1]$ is a solution of
$\frac{dx}{d\tau}=DK(x,t)=Dg(t)$, but it has no derivative at any point in the interval $[0,1]$.
Additionally, he pointed out that the GODE is a formal equation-like object for which one has defined its solutions.

   The concept of GODEs was initiated in Kurzweil  \cite{K-CMJ1957,K-CMJ1958}, who introduced it
 to generalize the classical results on the continuous dependence of the solutions of  ODEs with respect to parameters in 1957.
   Later, the fundamental theory in the framework of GODEs was established, one can consult Schwabik \cite{S-BOOK1992,S-MB1996,S-MB1999}.
  More recently, the stability theory and the qualitative property of GODEs are developed by many scholars.
  For example,  the variation of constants formula for the GODEs was proposed by Collegari et al. \cite{CFF-CMJ},
  the boundedness of solutions for GODEs were derived by Afonso et al. \cite{ABFG-MN2012} and Federson et al. \cite{F-JDE2017},
 the concepts of exponential dichotomy and its robustness results were given by Bonotto et al. \cite{BFS-JDE2018,BFS-JDDE2020},
 the results on topological properties of flows for nonnegative time in the framework of GODEs were also presented by Bonotto et al. \cite{B-JDE2021},
 the existence of periodic solutions for autonomous GODEs can be seen Federson et al. \cite{FGMT-N2022},
  the converse Lyapunov theorem for GODEs were described by Andrade da Silva et al. \cite{AMT-JDE2022} and
the boundary value problems for GODEs was obtained from Bonotto et al. \cite{BFM-JGA2023}.
      Other concepts on the GODEs were considered in the monograph of Bonotto et al. \cite{BFM-BOOK2021}.
 The GODEs include various types of other differential equations as special cases, such as the classical ODEs, measure differential equations (MDEs),
  impulsive differential equations (IDEs), functional differential equations (FDEs) as well as
  dynamic equations on time scales.
In particular, a pretty application of GODEs is MDEs, which have well studied (see e.g., Federson and Mesquita \cite{F-MB2002}, Federson, et al. \cite{FMS-JDE2012}, Piccoli \cite{P-ARMA}, Piccoli and Rossi \cite{P-DCDS}, Meng \cite{MG-PAMS}, Meng and Zhang \cite{MG-JDE}, Zhang \cite{ZhangMR-SCM}, Wen \cite{Wen-JDE}, Wen and Zhang \cite{Wen-DCDSB}, Chu and Meng \cite{chu-MA}, Chu, et al. \cite{Chu-JDE}, Liu, Shi, Yan \cite{YanJ-SCM}), another application of GODEs is IDEs (see Federson and Schwabik \cite{FS-DIE2006}, Afonso, et al. \cite{ABFS-JDE2011}).

\subsection{History of  linearization}
   Linearization, as an important topic in dynamical systems, portrays  the nonlinear perturbation systems through the dynamical behavior of
the linear systems.
   A fundamental contributions to this project in the autonomous systems is the Hartman-Grobman theorem  (see \cite{Hartman,Grobman}),
which describes that $C^1$ hyperbolic diffeomorphism $G:\mathbb{R}^d\to \mathbb{R}^d$ can be $C^0$ linearized near the fixed point.
   Later, Palis  \cite{Palis} and Pugh \cite{Pugh1} generalized the local and global Hartman-Grobman theorem to Banach space, respectively.
   In addition, infinite dimensional versions of this results were well presented by Lu \cite{Lu1} (scalar reaction-diffusion equation), Bates and Lu \cite{Lu2} (Cahn-Hilliard equation and phase field equations),
Hein and Pr\"{u}ss \cite{Hein-Pruss1} (semilinear hyperbolic evolution equations), Farkas \cite{Farkas} (retarded functional equations).
%Reinfelds and Sermone \cite{Reinfelds2},
Except for the $C^0$ linearization of the differential equations,
     Sternberg (\cite{Sternberg1,Sternberg2}) initially investigated $C^r$ linearization for $C^k$ diffeomorphisms.
Sell (\cite{Sell1}) extended the theorem of Sternberg. %He studied a diffeomorphism in the vicinity of a hyperbolic fixed point.
     Casta\~{n}eda and Robledo \cite{CR-JDE} formulated some sufficient conditions for smooth linearization of nonautonomous differential equations whose linear part admits
a uniform exponential contraction.
   Cuong et al. \cite{CDS-JDDE} proved a Sternberg-type theorem for nonautonomous systems
while the linear part has a uniform exponential dichotomy.
  Dragi\v{c}evi\'{c} et al. \cite{DZZ-MZ,DZZ-PLMS} showed discrete  and continuous smooth linearization results with nonuniform exponential dichotomy, respectively.
Some mathematicians paid particular attention on the $C^1$ linearization.
 Belitskii \cite{Belitskii1,Belitskii2},
ElBialy \cite{ElBialy1}, Rodrigues and Sol\`a-Morales \cite{RS-JDE} studied that
$C^1$ linearization of hyperbolic diffeomorphisms on Banach space independently.
%   Recently, Zhang and Zhang \cite{} proved that the    $C^1$ linearization is H\"{o}lder continuous. %A counter example was given to show that those estimates of   H\"{o}lder exponent cannot be improved anymore.
 Recently, Zhang et al. \cite{ZWN-JFA,ZWN-MA,ZWN-JDE,ZWM-TAMS}  showed the $C^{1,\beta}$-linearization for $C^{1,\alpha}$ or $C^{1,1}$ hyperbolic diffeomorphisms (where $0<\beta<\alpha\leq 1$), and they proved that the regularity for the transformations is sharp.

 Usually, the results of $C^1$ linearization are local. However, to describe the global dynamics,
many scholars paid attention to the global linearization of the non-autonomous systems, which was originated from Palmer  \cite{Palmer1}.
   Palmer established the global topological conjugacy between the nonautonomous nonlinear system and its linear part by constructing conjugate maps,
which also relies on the exponential dichotomy theory (see Coppel  \cite{Coppel-book}) for nonautonomous linear systems.
   After that, by weakening the assumption of exponential dichotomy, Jiang  \cite{Jiang2} obtained a version of Hartman-Grobman theorem for ordinary dichotomy.
Backes and Dragi\v{c}evi\'{c} \cite{BD-CM}  proved a linearization result for generalized exponential dichotomy, which extended the result of
Bernardes and Messaoudi \cite{BM-PAMS} from autonomous to nonautonomous systems.
   As well as Barreira and Valls \cite{B-V2} presented linearization results with the nonuniform exponential dichotomy.
   Reinfelds and \v{S}teinberga  \cite{Reinfelds-IJPAM} firstly studied the $C^0$ linearization for non-hyperbolic systems,
and Backes et al.  \cite{BDK-JDE} extended it to the non-hyperbolic coupled systems.
   By reducing the condition of boundedness for the nonlinear perturbations,
Xia et al.  \cite{Xia-BSM} reported a Hartman-Grobman theorem under locally integrable conditions.
  Casta\~{n}eda and Robledo \cite{CR-DCDSA} and Huerta \cite{Huerta} considered the linearization of the unbounded nonlinear system
under the nonuniform exponential contraction,  respectively.
   Backes and Dragi\v{c}evi\'{c} \cite{BD-arXiv1} gave a version of multiscale linearization.
   Qadir \cite{SIGMA} presented a geometric linearization of second order semilinear ordinary differential equations.
   Furthermore, $C^0$-linearization theory was investigated in various dynamical systems, for instance,
the functional differential equations (see Farkas \cite{Farkas}),
the classical impulsive differential equations (see Reinfelds and Sermone  \cite{Reinfelds2}, Reinfelds \cite{Rein1997,Rein2000},  Sermone \cite{Sermone3,Sermone4},
Fenner and Pinto  \cite{Fenner-Pinto} and Xia et al. \cite{Xia1}),
dynamic equations on time scales (see P\"otzche \cite{Potzche1} and Xia et al. \cite{Xia2}),
the differential equations with piecewise constant argument (see Papaschinopoulos  \cite{Papa-A}, Pinto and Robledo \cite{Pinto-JAA}).

\subsection{Motivation}

  The theory of linear GODEs has been well established, see the monographs of Schwabik \cite{S-BOOK1992} and Bonotto et al. \cite{BFM-BOOK2021}.
   However the theory of nonlinear GODEs is not yet mature. It is very important to study the relationships between the linear systems and the nonlinear systems. Hartman-Grobman theorem builds bridges between the linear system and its nonlinear perturbations.
    Up till now, there is no papers considering the topological conjugacy between the linear GODEs and their nonlinear perturbations. In this paper, we firstly establish the Hartman-Grobman-type theorem in the framework of GODEs based on the exponential dichotomy proposed by Bonotto et al.  \cite{BFS-JDE2018}.

   Consider the linear GODEs
\begin{equation}\label{1}
  \frac{dx}{d\tau}=D[A(t)x]
\end{equation}
and their nonlinear perturbation
\begin{equation}\label{2}
  \frac{dx}{d\tau}=D[A(t)x+F(x,t)]
\end{equation}
on Banach space $\mathscr{X}$, where $A:\mathbb{R}\to\mathscr{B}(\mathscr{X})$ is a bounded linear operator on $\mathscr{X}$ and $F:\mathscr{X}\times \mathbb{R}\to \mathscr{X}$ is Kurzweil integrable.
%Bonotto, Federson and Santo \cite{BFS-JDE2018} presented the concept of exponential dichotomy for Eq. \eqref{1}.  They further studied the existence and uniqueness of bounded solutions of the linear inhomogeneous GODEs
%\begin{equation*}
%  \frac{dx}{d\tau}=D[A(t)x+g(t)],
%\end{equation*}
%where $g:\mathbb{R}\to \mathscr{X}$ is a regulated function.
Note that the notation $\frac{dx}{d\tau}$ in \eqref{1} (or \eqref{2}) does not indicate that the solution has a derivative because it is only a symbol.  The solution of the GODEs can be discontinuous and even the number of discontinuous points is countable,
so that many classical theorems and tools are no longer applicable to the GODEs.
  For instances, the chain rule and the multiplication rule of derivatives, differential mean value theorem, and integral mean value theorem  are not valid for the GODEs. Due to the great differences between GODEs and ODEs, it is difficult to extend the Hartman-Grobman theorem of ODEs to the GODEs.

  In this paper, we prove the topological conjugacy between GODEs \eqref{2} and its linear part \eqref{1}. More precisely,
  constructing the approximately identical maps and using the relationship between the integral equations,
we verify that these maps are homeomorphisms and achieve the bijection between the solutions of Eqs. \eqref{1} and \eqref{2}.
  Furthermore,  we  study the regularity of these maps (or conjugacies).
  We give a result for H\"{o}lder conjugacies with the help of the Gronwall-type inequality (in the Perron-Stieltjes integral sense) and other nontrivial
estimate techniques.
   As applications, we derive results on the Hartman-Grobman theorem for the measure differential equations and impulsive differential equations.%the MDEs and the IDEs.

To construct the approximately identical maps, it is necessary to obtain the explicit bounded solutions of \eqref{2} provided that the linear GODEs \eqref{1} admit an exponential dichotomy. We remark that it is difficult to obtain the explicit bounded solutions of Eq. \eqref{2}. For the classical ODEs, standardly, we use the fixed point theory to obtain the bounded solutions of the nonlinear perturbation systems if its homogeneous linear system admits an exponential dichotomy. However, the standard method used in the classical ODEs is not valid for the GODEs. Provided that the linear GODEs \eqref{1} admit an exponential dichotomy, we construct the formulas for bounded solution of Eq. \eqref{2}
 by using the technique of Picard's stepwise approximation. We obtain the following explicit expression for the bounded solution
of a class of nonlinear GODEs:
\begin{equation}\label{3}
  \begin{split}
x(t)=& \int_0^t DF(x(\tau),s)-\int_{-\infty}^t d_{\sigma}[\mathscr{V}(t)P\mathscr{V}^{-1}(\sigma)]\int_0^\sigma DF(x(\tau),s) \\
&+ \int_{t}^\infty d_{\sigma}[\mathscr{V}(t)(I-P)\mathscr{V}^{-1}(\sigma)] \int_0^\sigma DF(x(\tau),s),
\end{split}
\end{equation}
where $\mathscr{V}(t):=V(t,0)$ is the fundamental operator of Eq. \eqref{1} and $P$ is a projection, $I$ is an identity operator.
  Note that the term $ \int_0^t DF(x(\tau),s)$ is a form specific to the nonlinear GODEs, since the Kurzweil integral is involved.
  If the Kurzweil integral  reduces to the Riemann integral or the Lebesgue integral,
the expression \eqref{3} is consistent with the case of the classical ODEs, see Coppel \cite{Coppel-book}.

\subsection{Outline}
  The rest of the paper is arranged as follows.
  In Section 2, we give the preliminary results for GODEs.
  Subsections 2.1 and 2.2 introduce the regulated functions, bounded variation functions and Kurzweil integral.  Subsection 2.3 presents the fundamental theory for GODEs.
  And a concept of strong exponential dichotomy is given.
  In Section 3, we analyze the bounded solution for Eq. \eqref{2}.
  Section 4 presents results which characterize the topological conjugacy between Eq. \eqref{1} and Eq. \eqref{2}.
  In Section 5, we derive the H\"{o}lder continuity for the conjugacies.
  We present some useful inequalities and then give rigorous proofs in subsections 5.2 and 5.3, respectively.
  Finally, we apply the results of GODEs to MDEs and IDEs.
  %measure differential equations and the impulsive differential equations.

\section{Preliminaries for the GODEs}
\subsection{regulated and bounded variation functions}
   We say that a function $g:[{a_1},{a_2}]\to \mathscr{X}$ is said to be  $regulated$ if
\[g(t^+):=\lim\limits_{\tau\to t^+} g(\tau), \quad t\in[{a_1},{a_2}) \quad  \mathrm{and} \quad
g(t^{-}):=\lim\limits_{\tau \to t^-} g(\tau), \quad t\in({a_1},{a_2}],\]
where $(\mathscr{X}, \|\cdot\|)$ is a Banach space.
Define
\[
G([{a_1},{a_2}],\mathscr{X})=\{g:[{a_1},{a_2}]\to \mathscr{X}| g \;\mathrm{ is} \; \mathrm{a} \; \mathrm{regulated}  \; \mathrm{function} \;
\mathrm{with}  \; \sup_{t\in[{a_1},{a_2}]}\|g(t)\|<\infty\}
\]
and $\|g\|_\infty:=\sup_{t\in[{a_1},{a_2}]}\|g(t)\|$.
Then $(G([{a_1},{a_2}],\mathscr{X}),\|\cdot\|_\infty)$  is a Banach space,  (see \cite{Honig-book}, Theorem 3.6).

 We say that a finite point set  $D=\{t_0,t_1,\cdots, t_j\}\subset [{a_1},{a_2}]$ such that ${a_1}=t_0\leq t_1\leq \cdots\leq t_j={a_2}$ is a $division$ of $[{a_1},{a_2}]$.
 If $|D|$ is the number of subintervals $[t_{j-1},t_j]$ of a division $D$ of $[{a_1},{a_2}]$, then we write $D=\{t_0,\cdots,t_{|D|}\}$.
   Denote that $\mathscr{D}[{a_1},{a_2}]$  is the set of all division of $[{a_1},{a_2}]$.
   A map $g:[{a_1},{a_2}]\to \mathscr{X}$ is said to be a $variation$  if
%   The $variation$ of a function $g:[{a_1},{a_2}]\to \mathscr{X}$ in $[{a_1},{a_2}]$ is defined by
\[
  \mathrm{var}_{{a_1}}^{{a_2}} g:=\sup_{D\in\mathscr{D}[{a_1},{a_2}]}\sum_{j=1}^{|D|}\|g(t_j)-g(t_{j-1})\|.
\]
  If $\mathrm{var}_{{a_1}}^{{a_2}}g<\infty$, then $g$ is a  $bounded$ $variation$ function on  $[{a_1},{a_2}]$.
  Define
\[
BV([{a_1},{a_2}],\mathscr{X})=\{g\in \mathscr{X}| g \; \mathrm{is} \; \mathrm{a} \; \mathrm{bounded} \; \mathrm{variation} \; \mathrm{function}\; \mathrm{with} \; \|g({a_1})\|+\mathrm{var}_{{a_1}}^{a_2} g<\infty\}
\]
and $\|g\|_{BV}:=\|g({a_1})\|+\mathrm{var}_{{a_1}}^{a_2} g$.
  Then $(BV([{a_1},{a_2}],\mathscr{X}),\|\cdot\|_{BV})$ is a Banach space. Further,
 $BV([{a_1},{a_2}],\mathscr{X}) \subset G([{a_1},{a_2}],\mathscr{X})$, see also \cite{Honig-book}.

%  If $\mathrm{var}_{{a_1}}^{{a_2}}g<\infty$, then $g$ is a function of $bounded$ $variation$ on the interval  $[{a_1},{a_2}]$.
%  We denote the space of all functions of bounded variation $g:[{a_1},{a_2}]\to \mathscr{X}$ by $BV([{a_1},{a_2}],\mathscr{X})$.
%  The space $BV([{a_1},{a_2}],\mathscr{X})$ equipped with the variation norm $\|g\|_{BV}:=\|g({a_1})\|+\mathrm{var}_{{a_1}}^{a_2} g$ is a Banach space,
%and $BV([{a_1},{a_2}],\mathscr{X}) \subset G([{a_1},{a_2}],\mathscr{X})$, see \cite{Honig-book}.

\subsection{Kurzweil integral}
 We  recall the concept of the  Kurzweil integral, which is defined by \cite{K-CMJ1957} and \cite{S-BOOK1992}.
  To begin it, we give some basic concepts.

 A $tagged$ $division$ of  $[{a_1},{a_2}]\subset \mathbb{R}$ is a finite collection of point-interval pairs $D=\{(\tau_j,[s_{j-1},s_j]): j=1,2,\cdots,|D| \}$,
 where ${a_1}=s_0\leq s_1 \leq \cdots \leq s_{|D|}={a_2}$ is a division of $[{a_1},{a_2}]$ and the tag $\tau_j \in [s_{j-1},s_j]$.
 %For each $j$, the element $\tau_j\in [s_{j-1},s_j]$ is named the tag of the interval $[s_{j-1},s_j]$.
A $gauge$ on $[{a_1},{a_2}]$ is an arbitrary  positive function $\varepsilon:[{a_1},{a_2}]\to (0,\infty)$.
A tagged division $D=\{(\tau_j,[s_{j-1},s_j]),j=1,2,\cdots,|D|\}$ of $[{a_1},{a_2}]$ is $\varepsilon$-fine for any gauge $\varepsilon$ on $[{a_1},{a_2}]$ if
\[
[s_{j-1},s_j]\subset (\tau_j-\varepsilon(\tau_j),\tau+\varepsilon(\tau_j))).
\]

%A $gauge$ on $[{a_1},{a_2}]$ is any positive function $\varepsilon:[{a_1},{a_2}]\to (0,\infty)$.
%Given a gauge $\varepsilon$ on $[{a_1},{a_2}]$, a tagged division $D=\{(\tau_j,[s_{j-1},s_j]),j=1,2,\cdots,|D|\}$ of $[{a_1},{a_2}]$ is $\varepsilon$-fine is
%\[
%[s_{j-1},s_j]\subset (\tau_j-\varepsilon(\tau_j),\tau+\varepsilon(\tau_j))).
%\]

\begin{definition}\label{KW} (Kurzweil integrable)
Assume that there exists a unique   $J\in \mathscr{X}$ such that the following conditions holds:
for any $\epsilon>0$,
there is a gauge $\varepsilon$ of $[{a_1},{a_2}]$ satisfying for each $\varepsilon$-fine tagged division $D=\{(\tau_j,[s_{j-1},s_j]), j=1,2,\cdots,|D| \}$ of $[{a_1},{a_2}]$,
we have
\[\|K(V,d)-J\|<\epsilon,\]
 where $K(V,d)=\sum_{j=1}^{|D|} [V(\tau_j,s_j)-V(\tau_j,s_{j-1})]$, and we write $J=\int_{a_1}^{a_2} DV(\tau,t)$ in this situation.
 Such function $V:[{a_1},{a_2}]\times [{a_1},{a_2}]\to \mathscr{X}$ is called Kurzweil integrable on $[{a_1},{a_2}]$.
\end{definition}

%We remark that the
%the Kurzweil integral is linear, additive with respect to adjacent intervals and it includes the known Perron-Stieltjes integral as well as its improper integrals, see
%[...] for more details.
%Furthermore,
%the Kurzweil integral can be extended to unbounded intervals, see [...] and [...].

\subsection{The fundamental theory for the GODEs}
We review the fundamental theory of GODEs.
  Let $\mathscr{B}(\mathscr{X})$ be the set of all bounded linear operators with the operator norm $\|\cdot\|$.
  Given  $\mathbb{I}\subseteq \mathbb{R}$ and  linear GODEs
\begin{equation}\label{LLL}
  \frac{dx}{d\tau}=D[A(t)x],
\end{equation}
where $A: \mathbb{I}\to \mathscr{B}(\mathscr{X})$.
As point out in  \cite{S-BOOK1992}, a function $x:[{a_1},{a_2}]\to\mathscr{X}$
is said to be  a solution of \eqref{LLL} iff
\begin{equation}\label{Z1}
   x({a_2})=x({a_1})+\int_{{a_1}}^{{a_2}} D[A(s)x(\tau)].
\end{equation}
  The integral on the right-hand side of \eqref{Z1} is a Kurzweil integral,
  which is denoted by the Perron-Stieltjes integral $\int_{a_1}^{a_2} d[A(s)]x(s)$ (see \cite{S-BOOK1992,S-MB1996}), since we represent $\int_{a_1}^{a_2} D[A(t)x(\tau)]$ as Riemann-Stieltjes sum taking the form of
$\sum [A(t_j)-A(t_{j-1}]x(\tau_j)$.

Let $I$ be an identity operator, $A(t^+)=\lim\limits_{s\to t+}A(s)$ and $A(t^-)=\lim\limits_{s\to t-}A(s)$, we suppose the following conditions throughout this paper.
\\
(H1) for any $[{a_1},{a_2}]\subset \mathbb{I}$,    $A\in BV([{a_1},{a_2}],\mathscr{B}(\mathscr{X}))$;\\
(H2)   for $t\in \mathbb{I} \backslash \sup(\mathbb{I})$,  $(I+[A(t^+)-A(t)])^{-1}\in \mathscr{B}(\mathscr{X})$ and for  $t\in \mathbb{I} \backslash \inf(\mathbb{I})$,
$(I-[A(t)-A(t^-)])^{-1}\in \mathscr{B}(\mathscr{X})$.\\
\indent The above assumptions guarantee the existence and uniqueness of the solution of \eqref{LLL}.

%Given an initial value $(x_0,t_0)\in \mathscr{X}\times [{a_1},{a_2}]$, we say that a solution of the initial value problem
%\begin{equation}\label{IVP}
%  \begin{cases}
%  \frac{dx}{d\tau}=D[A(t)x],\\
%  x(t_0)=x_0,
%  \end{cases}
%\end{equation}
%is given by $x(t)=x_0+\int_0^t D[A(s)x(\tau)]$ for $t\in[{a_1},{a_2}]$.

\begin{lemma}(\cite{S-MB1999}, Theorem 2.10)
Assume that (H1) and (H2) hold.
Given $(t_0,x_0)\in \mathbb{I}\times \mathscr{X}$.
Then the following linear GODE
\begin{equation}\label{Z2}
  \begin{cases}
  \frac{dx}{d\tau}=D[A(t)x(\tau)],\\
  x(t_0)=x_0,
  \end{cases}
\end{equation}
admits a unique solution on $\mathbb{I}$.
\end{lemma}

\begin{lemma}(\cite{CFF-CMJ}, Theorem 4.3)
An operator $V:\mathbb{I}\times \mathbb{I}\to \mathscr{B}(\mathscr{X})$ is said to be a fundamental operator of Eq. \eqref{LLL} if
\begin{equation}\label{EP}
  V(t,s)=I+\int_s^t d[A(r)]V(r,s), %\quad t,s\in \mathbb{I},
\end{equation}
and for any fixed $s\in \mathbb{I}$,  $V(\cdot,s)$ is of locally bounded variation in $\mathbb{I}$.
  Furthermore, the unique solution of \eqref{Z2}  is %given by
$x(t)=V(t,s)x_s$.
\end{lemma}

\begin{lemma}(\cite{CFF-CMJ}, Theorem 4.4)
  The operator $V:\mathbb{I}\times \mathbb{I} \to \mathscr{B}(\mathscr{X})$  has the following properties:\\
(1) $V(t,t)=I$;\\
(2) for any $[{a_1},{a_2}]\subset \mathbb{I}$, there is a positive constant $N>0$ satisfying
\[\begin{split}
&\|V(t,s)\|\leq N, \quad t,s\in[{a_1},{a_2}], \quad \mathrm{var}_{a_1}^{a_2} V(t,\cdot)\leq N,\quad t\in[{a_1},{a_2}],\\
&\mathrm{var}_{a_1}^{a_2} V(\cdot,s)\leq N,\quad s\in[{a_1},{a_2}];
\end{split}\]
(3) for any $t,r,s\in \mathbb{I}$, $V(t,s)=V(t,r)V(r,s)$;\\
(4) $V^{-1}(t,s)\in\mathscr{B}(\mathscr{X})$ and $V^{-1}(t,s)=V(s,t)$.
\end{lemma}

%We present the concept of the exponential dichotomy, which is given by \cite{BFS-JDE2018}.
\begin{definition}(exponential dichotomy \cite{BFS-JDE2018})
    We say that a linear GODEs \eqref{LLL} have an exponential dichotomy on $\mathbb{I}$
if there exist a projection $P:\mathscr{X}\to \mathscr{X}$ and constants $K,\alpha>0$ satisfying
\begin{equation}\label{ED}
  \begin{cases}
   \|\mathscr{V}(t)P\mathscr{V}^{-1}(s)\|\leq Ke^{-\alpha(t-s)} \quad \mathrm{for} \; t\geq  s,\\
   \| \mathscr{V}(t)(Id-P)\mathscr{V}^{-1}(s)\|\leq Ke^{\alpha(t-s)} \quad \mathrm{for} \; t< s,
  \end{cases}
\end{equation}
where $\mathscr{V}(t)=V(t,0)$ and $\mathscr{V}^{-1}(t)=V(0,t)$.
\end{definition}

\begin{definition}\label{s-ed}(strong exponential dichotomy)
   If the linear GODEs \eqref{LLL} have a strong exponential dichotomy on $\mathbb{I}$ if \eqref{ED} holds and there exists a constant $\widetilde{\alpha}\geq \alpha$ such that
\begin{equation}\label{SED}
 \| \mathscr{V}(t)\mathscr{V}^{-1}(s)\|\leq K e^{\widetilde{\alpha}(t-s)}, \quad \mathrm{for} \; t,s\in\mathbb{I}.
\end{equation}
\end{definition}

 Now we give the results on the perturbation theory of GODEs.
% \begin{lemma}(\cite{BFS-JDE2018}, Corollary 4.4)
%    Suppose that the linear GODE \eqref{LLL} admits an
%exponential dichotomy and  satisfies (H1)--(H2).
%    If $g\in G(\mathbb{R},\mathscr{X})$, then the perturbed GODE
%\begin{equation}\label{NHE}
%  \frac{dx}{d\tau}=D[A(t)x+g(t)]
%\end{equation}
%has at most one bounded solution.
%\end{lemma}

\begin{lemma}(\cite{BFS-JDE2018}, Proposition 4.5)
   Suppose that the linear homogenous GODEs \eqref{LLL} satisfy (H1)--(H2) and admit an
exponential dichotomy.    If $g\in G(\mathbb{R},\mathscr{X})$ and the Perron-Stieltjes integrals
 \begin{equation}\label{V3}
   \int_{-\infty}^t d_{\sigma}[\mathscr{V}(t)P\mathscr{V}^{-1}(\sigma)](g(\sigma)-g(0))
 \end{equation}
and
 \begin{equation}\label{V4}
   \int_{t}^\infty d_{\sigma}[\mathscr{V}(t)(I-P)\mathscr{V}^{-1}(\sigma)](g(\sigma)-g(0))
 \end{equation}
exist on $t\in\mathbb{R}$ and the maps \eqref{V3} and \eqref{V4} are bounded, then the non-homogenous GODEs
\begin{equation}\label{NHE}
  \frac{dx}{d\tau}=D[A(t)x+g(t)]
\end{equation}
 have a unique bounded solution.
\end{lemma}

  We observe that in Remark 4.11 from \cite{BFS-JDE2018}, if $f$ is bounded with $\|f(t)\|\leq M$ and
%there exists a positive constant $C>0$ such that
%$\|[I-(A(t^+)-A(t))]^{-1}\|\leq C$  ,$\|[I-(A(t)-A(t^-))]^{-1}\|\leq C$ and
\[
V_A:=\sup\{\mathrm{var}_a^b A: a,b\in\mathbb{R}, a<b\}<\infty,
\]
 then Bonotto et al. \cite{BFS-JDE2018} proved the existence of \eqref{V3}
and \eqref{V4}. Furthermore, they obtained
\begin{equation}\label{aaa}
\sup_{t\in\mathbb{R}}\left\|  \int_{-\infty}^t d_{\sigma}[\mathscr{V}(t)P\mathscr{V}^{-1}(\sigma)](f(\sigma)-f(0))\right\| \leq 2MK\|P\|C^3 e^{3CV_A}V_A^2
\end{equation}
and
\begin{equation}\label{bbb}
  \sup_{t\in\mathbb{R}}\left\| \int_{t}^\infty d_{\sigma}[\mathscr{V}(t)(I-P)\mathscr{V}^{-1}(\sigma)](f(\sigma)-f(0))\right\| \leq 2MK(1+\|P\|)C^3 e^{3CV_A}V_A^2.
\end{equation}

\begin{lemma}\label{changshu}(\cite{CFF-CMJ}, Theorem 4.10)
 Assume that  (H1) and (H2) hold.
 If $F:\mathscr{X} \times [{a_1},{a_2}]\to \mathscr{X}$ is  Kurzweil integrable, %(in the case where $V(\tau,t)=F(z(\tau),t)$),
  $[\tilde{{a_1}},\tilde{{a_2}}]\subseteq [{a_1},{a_2}]$
and $t_0\in [\tilde{{a_1}},\tilde{{a_2}}]$, then the GODEs %$z:[\tilde{c},\tilde{d}]\to \mathscr{Z}$ is a solution of
\begin{equation*}
  \begin{cases}
   \frac{dx}{d\tau}=D[A(t)x+F(x,t)],\\
   x(t_0)=x_0,
  \end{cases}
\end{equation*}
are equivalent to the following integral equations
%iff it is a solution of the following integral equation
\[
x(t)=V(t,t_0)x_0+\int_{t_0}^tDF(x(\tau),\gamma)-\int_{t_0}^t d_{\sigma}[V(t,\sigma)]\left(\int_{t_0}^{\sigma}DF(x(\tau),\gamma) \right).
\]
%where $V(\cdot,\cdot)$ is given by \eqref{EP}.
%
%  Assume that conditions (H1) and (H2) hold.
%  If $F:\mathscr{X}\times [\alpha,\beta]\to \mathscr{X}$ is  Kurzweil integrable (in the case where $V(\tau,t)=F(x(\tau),t)$), $[a,b]\subseteq [\alpha,\beta]$
%and $t_0\in [a,b]$, then $x:[a,b]\to \mathscr{X}$ is a solution of
%\begin{equation*}
%  \begin{cases}
%   \frac{dx}{d\tau}=D[A(t)x+F(x,t)],\\
%   x(t_0)=x_0,
%  \end{cases}
%\end{equation*}
%if and only if it is a solution of the integral equation
%\[
%x(t)=V(t,t_0)x_0+\int_{t_0}^tDF(x(\tau),s)-\int_{t_0}^t d_{\sigma}[V(t,\sigma)]\left(\int_{t_0}^{\sigma}DF(x(\tau),s) \right), \quad t\in[a,b],
%\]
%where $V(\cdot,\cdot)$ is given by \eqref{EP}.
\end{lemma}

We then define  a special class of functions $F:\mathscr{X}\times\mathbb{I}\to \mathscr{X}$. For convenience, we write $\Omega:=\mathscr{X}\times\mathbb{I}$.
\begin{definition}\cite{BFS-JDE2018}
Given a function $h:I \to \mathbb{R}$ is nondecreasing.
 We say that a function $F:\Omega\to \mathscr{X}$ belongs to the class $\mathscr{F}(\Omega,h)$ if
\begin{equation}\label{HHH1}
  \|F(x,t_2)-F(x,t_1)\|\leq |h(t_2)-h(t_1)|
\end{equation}
for any $(x,t_2)$ and $(x,t_1)\in\Omega$ and
\begin{equation}\label{HHH2}
  \|F(x,t_2)-F(x,t_1)-F(z,t_2)+F(z,t_1)\|\leq \|x-z\||h(t_2)-h(t_1)|
\end{equation}
for any $(x,t_2)$, $(x,t_1)$, $(z,t_2)$ and $(z,t_2)\in\Omega$.
\end{definition}

\section{The formulas for bounded solutions of the nonlinear GODEs}

   In the present paper, we consider the following nonlinear GODEs
\begin{equation}\label{NL}
  \frac{dx}{d\tau}=D[A(t)x+F(x,t)],
\end{equation}
where $A:\mathbb{R}\to \mathscr{B}(\mathscr{X})$ is a bounded linear operator and $F: \mathscr{X}\times \mathbb{R}\to \mathscr{X}$ is Kurzweil integrable.
  Furthermore, we make the following assumptions on $A$ and $F$:\\
(A1) suppose that \eqref{LLL} admits an exponential dichotomy;\\
(A2) there exists a positive constant $C>0$ such that $\|[I-(A(t)-A(t^-))]^{-1}\|\leq C$,
$\|[I-(A(t^+)-A(t))]^{-1}\|\leq C$ and
\[
V_A:=\sup\{\mathrm{var}_a^b A: a,b\in\mathbb{R}, a<b\}<\infty;
\]
(A3) the function $F\in\mathscr{F}(\Omega,h)$, where $h:\mathbb{R}\to \mathbb{R}$ is a nondecreasing function such that
\[
V_h:=\sup\{\mathrm{var}_a^b h: a,b\in\mathbb{R}, a<b\}<\infty.
\]

Then we state the following result in this section. %, which states the nonlinear GODE  has a unique bounded solution.
\begin{theorem}\label{Bound1}
  If conditions (A1)--(A3) hold, then the nonlinear GODEs \eqref{NL} have a unique bounded solution, which is defined by
\[\begin{split}
x(t)=& \int_0^t DF(x(\tau),s)-\int_{-\infty}^t d_{\sigma}[\mathscr{V}(t)P\mathscr{V}^{-1}(\sigma)]\int_0^\sigma DF(x(\tau),s) \\
&+ \int_{t}^\infty d_{\sigma}[\mathscr{V}(t)(I-P)\mathscr{V}^{-1}(\sigma)] \int_0^\sigma DF(x(\tau),s).
\end{split} \]
\end{theorem}

\begin{remark}
   We note that different from the expression of bounded solution under the classical ODEs,
Theorem \ref{Bound1} presents the formula of bounded solution of nonlinear GODEs in the sense of Kurzweil integral.
  If the nonlinear equation \eqref{NL} is in a Riemann-integrable or Lebesgue-integrable environment,
then the formula of the bounded solution is
 \[
x(t)= \int_{-\infty}^t \mathscr{V}(t)P\mathscr{V}^{-1}(\sigma) F(x(\sigma),\sigma)d\sigma - \int_{t}^\infty \mathscr{V}(t)(I-P)\mathscr{V}^{-1}(\sigma) F(x(\sigma),\sigma)d\sigma.
 \]
Obviously,  in a Riemann-integrable or Lebesgue-integrable environment, the result is the same as that of the classical ODEs.
\end{remark}

\begin{remark}
    As pointed out by Schwabik \cite{S-BOOK1992},
the GODEs often do not involve differentiation. Therefore, it is indispensable to deal with the relationship between various integral equations throughout the proof.
\end{remark}

Before giving the proof, we present a proposition that plays an important role in Theorem \ref{Bound1} and subsequent linearization results.
%Before giving the proof, we present a proposition, which plays an important role in Theorem \ref{Bound1} and linearization results.

\begin{proposition}\label{Zero}
  Suppose that the linear homogeneous GODEs \eqref{LLL}  have an
exponential dichotomy. Then the GODEs \eqref{LLL} have no non-trivial bounded solutions.
\end{proposition}

\begin{proof}
Different from Bonotto et al. \cite{BFS-JDE2018} (see Proposition 4.3), we prove it by contradiction.
(i) If $x(t)=0$, then the result is obvious.
(ii) We now suppose that the bounded solution $x(t)\neq 0$.
  Let $\zeta=x(0)$. Notice
\[
x(t)=\mathscr{V}(t)P\zeta+\mathscr{V}(t)(I-P)\zeta.
\]
  Consider the case of $t\leq 0$, from the first inequality of \eqref{ED}, we have
\[\begin{split}
   \|P\zeta\|=& \|\mathscr{V}(0)P\zeta\|=\|\mathscr{V}(0)P\mathscr{V}^{-1}(t)\mathscr{V}(t)P\zeta\| \\
   \leq & \|\mathscr{V}(0)P\mathscr{V}^{-1}(t)\|\|\mathscr{V}(t)P\zeta\|\leq Ke^{\alpha t}\|\mathscr{V}(t)P\zeta\|,
\end{split}\]
thus,
\begin{equation}\label{V1}
  \|\mathscr{V}(t)P\zeta\|\geq K^{-1}e^{-\alpha t}  \|P\zeta\|.
\end{equation}
  On the other hand, we get (also using \eqref{ED})
\begin{equation}\label{V2}
\begin{split}
\|\mathscr{V}(t)(I-P)\zeta\|=& \|\mathscr{V}(t)(I-P)\mathscr{V}^{-1}(0)\mathscr{V}(0)(I-P)\zeta\| \\
\leq& \|\mathscr{V}(t)(I-P)\mathscr{V}^{-1}(0)\|\|\mathscr{V}^{-1}(0)\mathscr{V}(0)(I-P)\zeta\| \leq K\|(I-P)\zeta\| e^{\alpha t}.
\end{split}
\end{equation}
Combining \eqref{V1} and \eqref{V2}, we have
\[\begin{split}
  \|x(t)\|=& \|\mathscr{V}(t)P\zeta+\mathscr{V}(t)(I-P)\zeta\|\\
  \geq& \|\mathscr{V}(t)P\zeta\|-\|\mathscr{V}(t)(I-P)\zeta\| \\
  \geq& K^{-1}e^{-\alpha t}  \|P\zeta\|-Ke^{\alpha t}\|(I-P)\zeta\|,
\end{split}\]
which implies that
\[
\lim\limits_{t\to -\infty}\|x(t)\|=\infty,
\]
and thus $x(t)$ is unbounded solution, which contradicts to the assumption.
  Hence, $x(t)=0$ for all $t\leq 0$.
   A similar argument can be applied to $t\geq 0$, then $x(t)=0$.
\end{proof}

\begin{proof}[Proof of Theorem \ref{Bound1}.]
{\bf Step 1.}  We claim that there exists a bounded solution satisfying Eq. \eqref{NL}.
Indeed, let $x_0 (t):=0$. Then
\[\begin{split}
x_1(t):=& \int_0^t DF(x_0(\tau),s)-\int_{-\infty}^t d_{\sigma}[\mathscr{V}(t)P\mathscr{V}^{-1}(\sigma)]\int_0^\sigma DF(x_0(\tau),s) \\
&+ \int_{t}^\infty d_{\sigma}[\mathscr{V}(t)(I-P)\mathscr{V}^{-1}(\sigma)] \int_0^\sigma DF(x_0(\tau),s).
\end{split} \]
In view of the definition of Kurzweil integral, for each $\varepsilon>0$, there exists a gauge $\varepsilon$ of $[0,t]$ such that for every $\varepsilon$-fine tagged division $D=\{(\tau_j,[s_{j-1},s_j]), j=1,2,\cdots,|D|\}$ of $[0,t]$, we have
\[
\left\|\int_0^t DF(x_0(\tau),s)\right\|=\left\|\sum_{j=1}^{|D|} \left[F(x(\tau_j),s_j)-F(x(\tau_j),s_{j-1})\right]\right\|.
\]
It follows form  condition (A3) that
\begin{equation}\label{LV}
 \left\|\sum_{j=1}^{|D|} \left[F(x(\tau_j),s_j)-F(x(\tau_j),s_{j-1})\right]\right\| \leq \left| \sum_{j=1}^{|D|} h(s_j)-h(s_{j-1})\right|\leq |h(t)-h(0)|\leq 2V_h.
\end{equation}
Then $x_1(t)$ is well defined, since
\[\begin{split}
 \|x_1(t)\|\leq & |h(t)-h(0)|+\int_{-\infty}^t \|d_{\sigma}[\mathscr{V}(t)P\mathscr{V}^{-1}(\sigma)]\| |h(\sigma)-h(0)| \\
 &+\int_{t}^\infty \|d_{\sigma}[\mathscr{V}(t)(I-P)\mathscr{V}^{-1}(\sigma)]\| |h(\sigma)-h(0)|,
\end{split}\]
which  together with  conditions (A1)--(A2) and \eqref{aaa}, \eqref{bbb} yields that
\[\begin{split}
 \|x_1(t)\|\leq 2V_h +2K\|P\|C^3e^{3CV_A}V_A^2 V_h+2K(1+\|P\|)C^3e^{3CV_A}V_A^2 V_h.
\end{split}\]
Taking $t=s$, by \eqref{ED}, we have that $\|P\|\leq K$ and thus $\|x_1(t)\|<\infty$, namely,  $x_1(t)$ is bounded and well-defined.
If for any fixed $m\in\mathbb{N}$, $x_m(t)$ is well defined and bounded, then
\[\begin{split}
  x_{m+1}(t):=& \int_0^t DF(x_m(\tau),s)-\int_{-\infty}^t d_{\sigma}[\mathscr{V}(t)P\mathscr{V}^{-1}(\sigma)]\int_0^\sigma DF(x_m(\tau),s) \\
&+ \int_{t}^\infty d_{\sigma}[\mathscr{V}(t)(I-P)\mathscr{V}^{-1}(\sigma)] \int_0^\sigma DF(x_m(\tau),s)
\end{split}\]
and $x_{m+1}(t)$ is also bounded. By the induction principle, for any $m\in\mathbb{N}$, the function sequence $\{x_m(t)\}_{m=0}^\infty$ is bounded.
Moreover,
\[\begin{split}
\|x_{m+1}(t)-x_m(t)\|\leq& \int_0^t \|DF(x_m(\tau),s)-DF(x_{m-1}(\tau),s)\|\\
&+\int_{-\infty}^t \|d_{\sigma}[\mathscr{V}(t)P\mathscr{V}^{-1}(\sigma)]\|
\int_0^\sigma \|DF(x_m(\tau),s)-DF(x_{m-1}(\tau),s)\| \\
&+   \int_{t}^\infty \|d_{\sigma}[\mathscr{V}(t)(I-P)\mathscr{V}^{-1}(\sigma)]\| \int_0^\sigma \|DF(x_m(\tau),s)-DF(x_{m-1}(\tau),s)\|.
\end{split}\]
From the definition of Kurzweil integral and the condition (A3), it follows that
\begin{equation}\label{Bound2}
\begin{split}
  &\left\|\int_0^t DF(x(\tau),s)-DF(y(\tau),s)\right\|\\
  =&
\left\|\sum_{j=1}^{|D|}\left[F(x(\tau_j),s_j)-F(x(\tau_j),s_{j-1})\right]-\left[F(y(\tau_j),s_j)-F(y(\tau_j),s_{j-1})\right]\right\|  \\
\leq& \sum_{j=1}^{|D|}\|x(\tau_j)-y(\tau_j)\|\cdot |h(s_j)-h(s_{j-1})| \\
=& \int_0^t \|x(\tau)-y(\tau)\|dh(s).
\end{split}
\end{equation}
By using condition (A3), we have
\[\begin{split}
\|x_{m+1}(t)-x_m(t)\|\leq& \int_0^t \|x_m(\tau)-x_{m-1}(\tau)\|dh(s) \\
&+\int_{-\infty}^t \|d_{\sigma}[\mathscr{V}(t)P\mathscr{V}^{-1}(\sigma)]\|\int_0^\sigma \|x_m(\tau)-x_{m-1}(\tau)\|dh(s) \\
&+ \int_{t}^\infty \|d_{\sigma}[\mathscr{V}(t)(I-P)\mathscr{V}^{-1}(\sigma)]\| \int_0^\sigma\|x_m(\tau)-x_{m-1}(\tau)\|dh(s).
\end{split}\]
Set $T_m:=\sup\limits_{t\in\mathbb{R}}\|x_{m+1}(t)-x_{m}(t)\|$.
  We derive from conditions (A1)-(A2) and \eqref{aaa}, \eqref{bbb} that
\[\begin{split}
T_m\leq&  |h(t)-h(0)|\cdot T_{m-1}+K^2 C^3e^{3CV_A}V_A^2 |h(t)-h(0)|\cdot T_{m-1}\\
 &+K(1+K)C^3e^{3CV_A}V_A^2 |h(t)-h(0)|\cdot T_{m-1}.
\end{split}\]
Taking $V_h$ sufficiently small. Then there exist a positive constant $\delta<1$ such that $T_m\leq \delta \cdot T_{m-1}$.
Hence, $\sum_{m=1}^{\infty}\|x_{m+1}(t)-x_{m}(t)\|$  converges uniformly on $\mathbb{R}$, which means that the function sequence
     $\{x_{m}(t)\}_{m=0}^\infty$ also converges uniformly on $\mathbb{R}$. We write
\[
\lim\limits_{m\to\infty}x_m(t)=\widetilde{x}(t),
\]
thus $\widetilde{x}(t)$ is bounded and
\begin{equation}\label{BB1}
  \begin{split}
\widetilde{x}(t)=& \int_0^t DF(\widetilde{x}(\tau),s)-\int_{-\infty}^t d_{\sigma}[\mathscr{V}(t)P\mathscr{V}^{-1}(\sigma)]\int_0^\sigma DF(\widetilde{x}(\tau),s) \\
&+ \int_{t}^\infty d_{\sigma}[\mathscr{V}(t)(I-P)\mathscr{V}^{-1}(\sigma)] \int_0^\sigma DF(\widetilde{x}(\tau),s).
\end{split}
\end{equation}
{\bf Step 2.}  We claim the uniqueness of the solution for GODE \eqref{NL}.
Let $y(t)$ be another bounded solution of GODE \eqref{NL} satisfying the initial value $y(0) = y$.
From Lemma \ref{changshu}, we have
\[\begin{split}
 y(t)=&\mathscr{V}(t)y+\int_0^t DF(y(\tau),s)-\int_0^t d_\sigma[\mathscr{V}(t)\mathscr{V}^{-1}(\sigma)]\int_0^\sigma DF(y(\tau),s)\\
 =&\mathscr{V}(t)y+\int_0^t DF(y(\tau),s)-\int_0^t d_\sigma[\mathscr{V}(t)P\mathscr{V}^{-1}(\sigma)]\int_0^\sigma DF(y(\tau),s) \\
 &\pm \int_{-\infty}^t d_\sigma[\mathscr{V}(t)P\mathscr{V}^{-1}(\sigma)]\int_0^\sigma DF(y(\tau),s)\\
 &-\int_0^t d_\sigma[\mathscr{V}(t)(I-P)\mathscr{V}^{-1}(\sigma)]\int_0^\sigma DF(y(\tau),s) \\
 &\pm \int_{t}^\infty d_\sigma[\mathscr{V}(t)(I-P)\mathscr{V}^{-1}(\sigma)]\int_0^\sigma DF(y(\tau),s)\\
 =& \mathscr{V}(t)y+\int_0^t DF(y(\tau),s)- \int_{-\infty}^t d_\sigma[\mathscr{V}(t)P\mathscr{V}^{-1}(\sigma)]\int_0^\sigma DF(y(\tau),s)\\
 &+ \int_{t}^\infty d_\sigma[\mathscr{V}(t)(I-P)\mathscr{V}^{-1}(\sigma)]\int_0^\sigma DF(y(\tau),s)\\
 &+\mathscr{V}(t)\left( \int_{-\infty}^0 d_\sigma[P\mathscr{V}^{-1}(\sigma)]-\int_0^\infty d_\sigma[(I-P)\mathscr{V}^{-1}(\sigma)]\right)\int_0^\sigma DF(y(\tau),s).
\end{split} \]
Set
\begin{equation}\label{w1}
y_1:=\int_{-\infty}^0 d_\sigma[P\mathscr{V}^{-1}(\sigma)]\int_0^\sigma DF(y(\tau),s)
\end{equation}
and
\begin{equation}\label{w2}
y_2:=\int_0^\infty d_\sigma[(I-P)\mathscr{V}^{-1}(\sigma)]\int_0^\sigma DF(y(\tau),s).
\end{equation}
Then it is not difficult to show that the above two integrals are well defined.
Note that $y(t)$ is bounded, and
since by using \eqref{LV} and conditions (A1)--(A3),  the following expression
\[\begin{split}
& \int_0^t DF(y(\tau),s)- \int_{-\infty}^t d_\sigma[\mathscr{V}(t)P\mathscr{V}^{-1}(\sigma)]\int_0^\sigma DF(y(\tau),s)\\
 &+ \int_{t}^\infty d_\sigma[\mathscr{V}(t)(I-P)\mathscr{V}^{-1}(\sigma)]\int_0^\sigma DF(y(\tau),s)
\end{split}\]
is bounded. Then
$\mathscr{V}(t)(y+y_1+y_2)$
is also bounded, which implies that
\[
\mathscr{V}(t)(y+y_1+y_2)=0
\]
due to Proposition \ref{Zero}.
  Therefore,
\begin{equation}\label{BB2}
\begin{split}
y(t)=&\int_0^t DF(y(\tau),s)- \int_{-\infty}^t d_\sigma[\mathscr{V}(t)P\mathscr{V}^{-1}(\sigma)]\int_0^\sigma DF(y(\tau),s)\\
 &+ \int_{t}^\infty d_\sigma[\mathscr{V}(t)(I-P)\mathscr{V}^{-1}(\sigma)]\int_0^\sigma DF(y(\tau),s).
\end{split}
\end{equation}
 Combining \eqref{BB1} and \eqref{BB2}, we have
\begin{equation}\label{XY}
\begin{split}
\|\widetilde{x}(t)-y(t)\|\leq&  \int_0^t \|DF(\widetilde{x}(\tau),s)-DF(y(\tau),s)\|\\
&+\int_{-\infty}^t \|d_{\sigma}[\mathscr{V}(t)P\mathscr{V}^{-1}(\sigma)]\|
\int_0^\sigma \|DF(\widetilde{x}(\tau),s)-DF(y(\tau),s)\| \\
&+   \int_{t}^\infty \|d_{\sigma}[\mathscr{V}(t)(I-P)\mathscr{V}^{-1}(\sigma)]\| \int_0^\sigma \|DF(\widetilde{x}(\tau),s)-DF(y(\tau),s)\|,
\end{split}
\end{equation}
which implies
\[\sup_{t\in\mathbb{R}}\|\widetilde{x}(t)-y(t)\| \leq \delta \sup_{t\in\mathbb{R}}\|\widetilde{x}(t)-y(t)\|,\]
due to conditions (A1)--(A3) and \eqref{aaa}, \eqref{bbb}. Since $V_h$ is sufficiently small such that $\delta<1$, we conclude that $\widetilde{x}=y$.
Consequently, the uniqueness is claimed.
%The proof of the theorem is complete.
\end{proof}

\section{Topological conjugacy for the GODEs}
  In this section, we presents a version of  Hartman-Grobman-type theorem  in the framework of GODEs.

\begin{theorem}\label{linearization-thm}
Assume that  assumptions (A1)--(A3) hold. For the sufficiently small $V_h$,  the nonlinear GODEs \eqref{NL} are topologically conjugated to the linear GODEs \eqref{LLL}.
\end{theorem}

\begin{remark}
  Theorem \ref{linearization-thm} states the Hartman-Grobman theorem in the framework of GODEs.
   We construct two approximate identity maps $\Phi$ and $\Psi$,\\
(i) prove that $\Phi$ sends the solution of the linear GODEs onto the solution of  the nonlinear GODEs;\\
(ii) prove that $\Psi$   maps the solution of the nonlinear GODEs to the solution of the linear GODEs;\\
(iii) verify that $\Phi$ is a homeomorphism map and its inverse is $\Psi$, that is, $\Phi\circ \Psi=I$ and $\Psi\circ \Phi=I$.
Finally, the definition of topological conjugacy for nonautonomous systems proposed by Palmer \cite{Palmer1} is verified.
Different from Palmer, in the GODEs' environment, we  use the Kurzweil integral theory to deal with the relationship between more complicated integral equations.
\end{remark}

\begin{proof}
{\bf Step 1.} We establish the existence of the map $\Phi$, where $\Phi(t,x):=x+\phi(t,x)$.
Let $\Theta$ be the space of all maps $\phi:\mathbb{R}\times \mathscr{X}\to \mathscr{X}$ such that
\[
\|\phi\|_{\infty}:=\sup\limits_{t,x}\|\phi(t,x)\|<\infty.
\]
  Then $(\Theta,\|\cdot\|_\infty)$ is a Banach space. Given $\phi\in\Theta$, we define an operator $\mathscr{T}:\Theta\to \Theta$ as follows
\[\begin{split}
(\mathscr{T}\phi)(\tau,\xi)=& \int_0^\tau DF(p(r),s)-\int_{-\infty}^{\tau}d_{\sigma}[\mathscr{V}(\tau)P\mathscr{V}^{-1}(\sigma)]\int_0^\sigma DF(p(r),s)\\
&+\int_{t}^{\infty}d_{\sigma}[\mathscr{V}(\tau)(I-P)\mathscr{V}^{-1}(\sigma)]\int_0^\sigma DF(p(r),s),
\end{split}\]
where
\begin{equation}\label{ppw}
  p(r)=\mathscr{V}(r)\mathscr{V}^{-1}(\tau)\xi+\phi(r,\mathscr{V}(r)\mathscr{V}^{-1}(\tau)\xi)
\end{equation}
 and $(\tau,\xi)\in \mathbb{R}\times \mathscr{X}$.
Similar to the procedure of \eqref{LV}, we have
\[\begin{split}
C_{\tau}:=& \left\| \int_0^\tau DF(p(r),s) \right\|=\left\| \sum_{j=1}^{|D|} [F(p(r_j),s_j)-F(p(r_j),s_{j-1})]\right\| \\
 =& |h(\tau)-h(0)|\leq 2V_h.
\end{split}\]
Thus,
\[
\|(\mathscr{T}\phi)(\tau,\xi)\|\leq C_\tau+\int_{-\infty}^{\tau}\|d_{\sigma}[\mathscr{V}(\tau)P\mathscr{V}^{-1}(\sigma)]\|\cdot C_\sigma
+\int_{\tau}^{\infty}\|d_{\sigma}[\mathscr{V}(\tau)(I-P)\mathscr{V}^{-1}(\sigma)]\|\cdot C_\sigma<\infty,\]
by Theorem \ref{Bound1}. This implies that $\mathscr{T}\phi\in\Theta$.
Set
\[q_i(r)=\mathscr{V}(r)\mathscr{V}^{-1}(\tau)\xi+\phi_i(r,\mathscr{V}(r)\mathscr{V}^{-1}(\tau)\xi),\quad i=1,2.\]
Then,  taking any $\phi_1,\phi_2\in \Theta$, and by using \eqref{Bound2} we have
\[\begin{split}
 &\|(\mathscr{T}\phi_1)(\tau,\xi)-(\mathscr{T}\phi_2)(\tau,\xi)\|\\
 \leq& \int_0^\tau\| DF(q_1(r),s)-DF(q_2(r),s)\| \\
 &+\int_{-\infty}^{\tau}\|d_{\sigma}[\mathscr{V}(\tau)P\mathscr{V}^{-1}(\sigma)]\|\int_0^\sigma \| DF(q_1(r),s)-DF(q_2(r),s)\| \\
 &+\int_{\tau}^{\infty}\|d_{\sigma}[\mathscr{V}(\tau)(I-P)\mathscr{V}^{-1}(\sigma)]\|\int_0^\sigma \| DF(q_1(r),s)-DF(q_2(r),s)\| \\
\leq&  \int_0^\tau \|\phi_1(r,\mathscr{V}(r)\mathscr{V}^{-1}(\tau)\xi)-\phi_2(r,\mathscr{V}(r)\mathscr{V}^{-1}(\tau)\xi)\|dh(s) \\
&+\int_{-\infty}^{\tau}\|d_{\sigma}[\mathscr{V}(\tau)P\mathscr{V}^{-1}(\sigma)]\|\int_0^\sigma \|\phi_1(r,\mathscr{V}(r)\mathscr{V}^{-1}(\tau)\xi)-\phi_2(r,\mathscr{V}(r)\mathscr{V}^{-1}(\tau)\xi)\|dh(s)\\
&+\int_{\tau}^{\infty}\|d_{\sigma}[\mathscr{V}(\tau)(I-P)\mathscr{V}^{-1}(\sigma)]\|\int_0^\sigma \|\phi_1(r,\mathscr{V}(r)\mathscr{V}^{-1}(\tau)\xi)-\phi_2(r,\mathscr{V}(r)\mathscr{V}^{-1}(\tau)\xi)\|dh(s).
\end{split}\]
It follows from (A1)--(A3) and \eqref{aaa}, \eqref{bbb} that
\[
\|\mathscr{T}\phi_1-\mathscr{T}\phi_2\|_\infty \leq
2V_h(1+K\|P\|C^3e^{3CV_A}V_A^2+K(1+\|P\|)C^3e^{3CV_A}V_A^2) \|\phi_1-\phi_2\|_\infty.\]
Taking $V_h$  sufficiently small satisfying
\[2V_h(1+K\|P\|C^3e^{3CV_A}V_A^2+K(1+\|P\|)C^3e^{3CV_A}V_A^2)<1.\]
Then we obtain that  the map $\mathscr{T}:\Theta \to \Theta$ is a contraction and consequently $\mathscr{T}$ has a unique fixed point $\phi \in\Theta$
satisfying $\mathscr{T}\phi=\phi$. Hence,
\[\begin{split}
\phi(\tau,\xi)=& \int_0^\tau DF(p(r),s)-\int_{-\infty}^{\tau}d_{\sigma}[\mathscr{V}(\tau)P\mathscr{V}^{-1}(\sigma)]\int_0^\sigma DF(p(r),s)\\
&+\int_{\tau}^{\infty}d_{\sigma}[\mathscr{V}(\tau)(I-P)\mathscr{V}^{-1}(\sigma)]\int_0^\sigma DF(p(r),s),
\end{split}\]
where $p(r)$ is given by \eqref{ppw}. By using the identity
\[
\mathscr{V}(t)\mathscr{V}^{-1}(r)\mathscr{V}(r)\mathscr{V}^{-1}(\tau)x=\mathscr{V}(t)\mathscr{V}^{-1}(\tau)x,
\]
we have
\[\begin{split}
& \int_0^t DF(\mathscr{V}(r)\mathscr{V}^{-1}(t)\mathscr{V}(t)\mathscr{V}^{-1}(\tau)x+
 \phi(r,\mathscr{V}(r)\mathscr{V}^{-1}(t)\mathscr{V}(t)\mathscr{V}^{-1}(\tau)x),s)\\
=& \int_0^t
 DF(\mathscr{V}(r)\mathscr{V}^{-1}(\tau)x+
 \phi(r,\mathscr{V}(r)\mathscr{V}^{-1}(\tau)x),s),
\end{split}\]
and thus
\[\begin{split}
&\phi(t,\mathscr{V}(t)\mathscr{V}^{-1}(\tau)x)\\
=& \int_0^t DF(\mathscr{V}(r)\mathscr{V}^{-1}(\tau)x+ \phi(r,\mathscr{V}(r)\mathscr{V}^{-1}(\tau)x),s) \\
&-\int_{-\infty}^t d_\sigma[\mathscr{V}(t)P\mathscr{V}^{-1}(\sigma)]\int_0^\sigma DF(\mathscr{V}(r)\mathscr{V}^{-1}(\tau)x+ \phi(r,\mathscr{V}(r)\mathscr{V}^{-1}(\tau)x),s) \\
&+\int_t^\infty d_\sigma[\mathscr{V}(t)(I-P)\mathscr{V}^{-1}(\sigma)]\int_0^\sigma DF(\mathscr{V}(r)\mathscr{V}^{-1}(\tau)x+ \phi(r,\mathscr{V}(r)\mathscr{V}^{-1}(\tau)x),s).
\end{split}\]
If $t\mapsto x(t)$ is a solution of \eqref{LLL}, then
\[\begin{split}
\phi(t,x(t))
=& \int_0^t DF(x(r)+ \phi(r,x(r)),s) -\int_{-\infty}^t d_\sigma[\mathscr{V}(t)P\mathscr{V}^{-1}(\sigma)]\int_0^\sigma DF(x(r)+ \phi(r,x(r)),s) \\
&+\int_t^\infty d_\sigma[\mathscr{V}(t)(I-P)\mathscr{V}^{-1}(\sigma)]\int_0^\sigma DF(x(r)+ \phi(r,x(r)),s).
\end{split}\]
We next prove that $\Phi(t,x(t)):=x(t)+\phi(t,x(t))$ sends the solution of the linear GODEs \eqref{LLL} onto the solution of the nonlinear GODEs \eqref{NL}.
In fact,
\[\begin{split}
 \Phi(t,x(t))=& \mathscr{V}(t)x+\int_0^t DF(x(r)+\phi(r,x(r)),s)\\
 & -\int_{-\infty}^t d_\sigma[\mathscr{V}(t)P\mathscr{V}^{-1}(\sigma)]\int_0^\sigma DF(x(r)+ \phi(r,x(r)),s) \\
 &+\int_t^\infty d_\sigma[\mathscr{V}(t)(I-P)\mathscr{V}^{-1}(\sigma)]\int_0^\sigma DF(x(r)+ \phi(r,x(r)),s)\\
 =& \mathscr{V}(t)x+\int_0^t DF(\Phi(r,x(r)),s)
 -\int_{-\infty}^t d_\sigma[\mathscr{V}(t)P\mathscr{V}^{-1}(\sigma)]\int_0^\sigma DF(\Phi(r,x(r)),s)\\
 &+\int_t^\infty d_\sigma[\mathscr{V}(t)(I-P)\mathscr{V}^{-1}(\sigma)]\int_0^\sigma DF(\Phi(r,x(r)),s)\\
 :=& \mathscr{V}(t)x+\int_0^t DF(\Phi(r,x(r)),s)-I_1+I_2.
\end{split}\]
We divide the integrals $I_1$ and $I_2$ into two parts:
\[ I_1=\int_0^t+\int_{-\infty}^0:=I_{11}+\mathscr{V}(t)\cdot I_{12}\quad \mathrm{and} \quad
I_2=\int_t^0+\int_0^{\infty}:=I_{21}+\mathscr{V}(t)\cdot I_{22}. \]
It is easy to obtain that
\[-I_{11}+I_{21}=
\int_t^0 d_\sigma[\mathscr{V}(t)\mathscr{V}^{-1}(\sigma)]\int_0^\sigma DF(\Phi(r,x(r)),s).
\]
Similar to \eqref{w1} and \eqref{w2}, we see that $I_{12}$ and $I_{22}$ are bounded, and we denote them by $x_1$ and $x_2$.
Consequently,
\[\begin{split}
 \Phi(t,x(t))=&\mathscr{V}(t)(x-x_1+x_2)+\int_0^t DF(\Phi(r,x(r)),s)-\int_0^t d_\sigma[\mathscr{V}(t)\mathscr{V}^{-1}(\sigma)]\int_0^\sigma DF(\Phi(r,x(r)),s),
\end{split}\]
which implies that $\Phi(t,x(t))$ is a solution of the nonlinear GODEs \eqref{NL}.
\\
{\bf Step 2.} We establish the existence of the map $\Psi$, where $\Psi(t,x):=x+\psi(t,x)$.
  Let $x(t,\tau,\xi)$ be the solution of the nonlinear GODEs \eqref{NL} with the initial value $(\tau,\xi)\in\mathbb{R}\times \mathscr{X}$.
  Set
\[\begin{split}
\psi(\tau,\xi)=& -\int_0^\tau DF(x(r,\tau,\xi),s)
   +\int_{-\infty}^\tau d_\sigma[\mathscr{V}(\tau)P\mathscr{V}^{-1}(\sigma)]\int_0^\sigma DF(x(r,\tau,\xi),s)\\
   &-\int_\tau^\infty d_\sigma[\mathscr{V}(\tau)(I-P)\mathscr{V}^{-1}(\sigma)]\int_0^\sigma DF(x(r,\tau,\xi),s).
\end{split}\]
  Similar to the procedure in $\phi$, we show that $\psi\in \Theta$.
  Then,
\[\begin{split}
 \psi(t,x(t,\tau,x))=& -\int_0^t DF(x(r,\tau,x),s)
   +\int_{-\infty}^t d_\sigma[\mathscr{V}(t)P\mathscr{V}^{-1}(\sigma)]\int_0^\sigma DF(x(r,\tau,x),s)\\
   &-\int_t^\infty d_\sigma[\mathscr{V}(t)(I-P)\mathscr{V}^{-1}(\sigma)]\int_0^\sigma DF(x(r,\tau,x),s)
 \end{split}\]
since the identity
\[ x(t,s,x(s,\tau,x))=x(t,\tau,x).\]
We next prove that $\Psi(t,x(t)):=x(t)+\psi(t,x(t))$ is a solution of the linear GODEs \eqref{LLL}.
 In fact,
\[\begin{split}
\Psi(t,x(t))=&\mathscr{V}(t)x+\int_0^t DF(x(r),s)-\int_0^t d_\sigma[\mathscr{V}(t)\mathscr{V}^{-1}(\sigma)]\int_0^\sigma DF((x(r),s)\\
&-\int_0^t DF(x(r),s)
   +\int_{-\infty}^t d_\sigma[\mathscr{V}(t)P\mathscr{V}^{-1}(\sigma)]\int_0^\sigma DF(x(r),s)\\
   &-\int_t^\infty d_\sigma[\mathscr{V}(t)(I-P)\mathscr{V}^{-1}(\sigma)]\int_0^\sigma DF(x(r),s)\\
:=& \mathscr{V}(t)x-\int_0^t d_\sigma[\mathscr{V}(t)\mathscr{V}^{-1}(\sigma)]\int_0^\sigma DF((x(r),s)+J_1-J_2.
 \end{split}\]
We also divide the integrals $J_1$ and $J_2$ into two parts:
\[ J_1=\int_0^t+\int_{-\infty}^0:=J_{11}+\mathscr{V}(t)\cdot J_{12}\quad \mathrm{and} \quad
J_2=\int_t^0+\int_0^{\infty}:=J_{21}+\mathscr{V}(t)\cdot J_{22}. \]
Then
\[J_{11}-J_{21}=
\int_0^t d_\sigma[\mathscr{V}(t)\mathscr{V}^{-1}(\sigma)]\int_0^\sigma DF(x(r),s),
\]
and $J_{12}$ and $J_{22}$ are bounded, and we denote them by $x_3$ and $x_4$. Hence,
\[
\Psi(t,x(t))=\mathscr{V}(t)(x+x_3-x_4),\]
which indicates that $\Psi(t,x(t))$ maps the solution of the nonlinear GODEs \eqref{NL} onto the solution of the linear GODEs \eqref{LLL}.
\\
{\bf Step 3.} This step is to prove $\Psi \circ \Phi =I$, that is, $\Psi(t,\Phi(t,x(t)))=x(t)$, where $x(t)$ is the solution of the linear GODEs \eqref{LLL}.
Indeed, by Step 1, one  sees that $\Phi(t,x(t))$ is the solution of the nonlinear GODEs \eqref{NL}, and by Step 2, $\Psi(t,\Phi(t,x(t)))$ is the solution of  the linear GODEs \eqref{LLL}.
We write $\widehat{x}(t)=\Psi(t,\Phi(t,x(t)))$. Set $L(t)=\widehat{x}(t)-x(t)$. Then $L(t)$ is also the solution of  the linear GODEs \eqref{LLL}.
Thus
\[\begin{split}
\|L(t)\|=& \|\Psi(t,\Phi(t,x(t)))-x(t)\| \leq \|\Psi(t,\Phi(t,x(t)))-\Phi(t,x(t))\|+\|\Phi(t,x(t))-x(t)\| \\
\leq& \|\psi(t,\Phi(t,x(t)))\|+\|\phi(t,x(t))\|<\infty,
\end{split}\]
since $\phi,\psi\in\Theta$. From Proposition \ref{Zero}, it follows that $L(t)=0$, namely, $\Psi(t,\Phi(t,x(t)))=x(t)$.
\\
{\bf Step 4.} We claim that $\Phi\circ \Psi=I$, that is, $\Phi(t,\Psi(t,x(t)))=x(t)$, where $x(t)$ is the solution of the nonlinear GODE \eqref{NL}.
  In fact, it follows from Step 2 and Step 1 that $\Psi(t,x(t))$ is the solution of the linear GODEs \eqref{LLL}
and $\Phi(t,\Psi(t,x(t)))$ is the solution of the nonlinear GODEs \eqref{NL}.
  We write $\widetilde{x}(t):=\Phi(t,\Psi(t,x(t)))$.
  Set $H(t)=\widetilde{x}(t)-x(t)$.
  Then $H(t)$ is the solution of the following integral equation
\begin{equation}\label{H1}
\begin{split}
 H(t)=&\mathscr{V}(t)(\widetilde{x}(0)-x(0))+\int_0^t (DF(H(r)+x(r),s)-DF(x(r),s))\\
 &-\int_0^t d_\sigma[\mathscr{V}(t)\mathscr{V}^{-1}(\sigma)]\int_0^\sigma (DF(H(r)+x(r),s)-DF(x(r),s))\\
 :=&K_1+K_2-K_3.
\end{split}
\end{equation}
In addition,
\[\begin{split}
\|H(t)\|=& \|\Phi(t,\Psi(t,x(t)))-x(t)\|\leq \|\Phi(t,\Psi(t,x(t)))-\Psi(t,x(t))\|+\|\Psi(t,x(t))-x(t)\|\\
\leq& \|\phi(t,\Psi(t,x(t)))\|+\|\psi(t,x(t))\|<\infty,
\end{split}\]
due to $\phi,\psi\in\Theta$.
Since
\[\begin{split}
K_3=& \int_0^t d_\sigma[\mathscr{V}(t)\mathscr{V}^{-1}(\sigma)]\int_0^\sigma (DF(H(r)+x(r),s)-DF(x(r),s))\\
=& \int_0^t d_\sigma[\mathscr{V}(t)P\mathscr{V}^{-1}(\sigma)]\int_0^\sigma (DF(H(r)+x(r),s)-DF(x(r),s))\\
&+ \int_0^t d_\sigma[\mathscr{V}(t)(I-P)\mathscr{V}^{-1}(\sigma)]\int_0^\sigma (DF(H(r)+x(r),s)-DF(x(r),s))\\
:=&K_{31}+\mathscr{V}(t)\cdot K_{32}+K_{33}+\mathscr{V}(t)\cdot K_{34},
\end{split}\]
where
\[\begin{split}
K_{31}=& \int_{-\infty}^t d_\sigma[\mathscr{V}(t)P\mathscr{V}^{-1}(\sigma)]\int_0^\sigma (DF(H(r)+x(r),s)-DF(x(r),s)), \\
K_{32}=& \int_0^{-\infty} d_\sigma[P\mathscr{V}^{-1}(\sigma)]\int_0^\sigma (DF(H(r)+x(r),s)-DF(x(r),s)),\\
K_{33}=&-\int_t^\infty d_\sigma[\mathscr{V}(t)(I-P)\mathscr{V}^{-1}(\sigma)]\int_0^\sigma (DF(H(r)+x(r),s)-DF(x(r),s)),\\
K_{34}=& \int_0^\infty d_\sigma[(I-P)\mathscr{V}^{-1}(\sigma)]\int_0^\sigma (DF(H(r)+x(r),s)-DF(x(r),s)).
\end{split}\]
Similar to \eqref{w1} and \eqref{w2}, the integrals $K_{32}$ and $K_{34}$ are bounded, we denote them by $x_5$ and $x_6$.
By using Theorem \ref{Bound1}, it is obvious that the integral $K_2+K_{31}+K_{33}$ is bounded.
It means that $K_1+\mathscr{V}(t)(x_5+x_6)$ is bounded, but $K_1+\mathscr{V}(t)(x_5+x_6)$ is a solution of the linear GODE \eqref{LLL}, which further implies
that $K_1+\mathscr{V}(t)(x_5+x_6)=0$ (by using Proposition \ref{Zero}).
Thus, \eqref{H1} can be written as
\[\begin{split}
H(t)=& \int_0^t (DF(H(r)+x(r),s)-DF(x(r),s))\\
&-
 \int_{-\infty}^t d_\sigma[\mathscr{V}(t)P\mathscr{V}^{-1}(\sigma)]\int_0^\sigma (DF(H(r)+x(r),s)-DF(x(r),s))\\
 &+\int_t^\infty d_\sigma[\mathscr{V}(t)(I-P)\mathscr{V}^{-1}(\sigma)]\int_0^\sigma (DF(H(r)+x(r),s)-DF(x(r),s))
\end{split}\]
Similar to \eqref{XY}, we have
\[
\sup\limits_{t\in\mathbb{R}}\|H(t)\|\leq \delta \sup\limits_{t\in\mathbb{R}}\|H(t)\|,
\]
where $\delta<1$. Therefore, $H(t)=0$, namely, $\Phi(t,\Psi(t,x(t)))=x(t)$.

Consequently, it follows from Steps 1--4 that the nonlinear GODEs \eqref{NL} is topologically conjugated to the linear GODEs \eqref{LLL}.
\end{proof}

\section{H\"{o}lder conjuacies}
 In this section, we consider the H\"{o}lder regularity for the conjugacies $\Phi$ and $\Psi$.
  We propose a strong hypothesis:
\begin{equation}\label{SFF}
  \|F(x,s)-F(x,t)\|\leq e^{-\alpha|s-t|}|h(s)-h(t)|, \quad \mathrm{for} \; t,s \in\mathbb{R} \; \mathrm{and} \; x\in \mathscr{X},
\end{equation}
where $\alpha $ is defined in \eqref{ED}.
  Notice that this hypothesis does not change our linearization results since
  \[e^{-\alpha|s-t|}|h(s)-h(t)|\leq |h(s)-h(t)|.\]
  Then we have the result of this section.
\begin{theorem}\label{holder-thm}
  Suppose that all conditions of Theorem \ref{linearization-thm} hold.
 If further $F(x.t)$ satisfies \eqref{SFF}, then \\
 (i) there exists a positive constant $C>0$ such that for any $t\in \mathbb{R}$ and $x,\widetilde{x}\in \mathscr{X}$ satisfying $0<\|x-\widetilde{x}\|<1$, we have
\[
\|\Psi(t,x)-\Psi(t,\widetilde{x})\|\leq C\|x-\widetilde{x}\|^{\frac{\alpha}{\alpha+\widetilde{\alpha}}},
\]
where $\widetilde{\alpha}$ is given by Definition \ref{s-ed};\\
(ii) there exist positive constants $\widehat{C}>0$  and $0<q\leq \frac{\alpha}{\alpha+\widetilde{\alpha}}$ such that for any $t\in \mathbb{R}$ and $x,\widetilde{x}\in \mathscr{X}$ satisfying $0<\|x-\widetilde{x}\|<1$, we have
\[
\|\Phi(t,x)-\Phi(t,\widetilde{x})\|\leq \widehat{C}\|x-\widetilde{x}\|^{q}.
\]
\end{theorem}

\begin{remark}
  We emphasize that condition \eqref{SFF} is to help us deal with the integral $\int_0^t DF(X(\tau),s)$ under Kurzweil integrable.
  If it is under Riemann integrable or Lebesgue integrable, condition \eqref{SFF} is not required, one can refer to Pinto and Robledo \cite{Pinto-JAA}.
\end{remark}

To prove it, we first present some useful inequalities and then give rigorous proofs in subsections 5.2 and 5.3, respectively.
 For the sake of convenience,
  given $t,t_0\in\mathbb{R}$ and $x,y\in X$,
  let $X(t,t_0,x)$ be the solution of Eq. \eqref{NL} satisfying the initial value $x(t_0)=x$, and let $Y(t,t_0,y)=\mathscr{V}(t)y$ be the solution of
Eq. \eqref{LLL} such that $Y(t_0)=y$.
\subsection{Some useful estimates}
\begin{lemma}\label{bellman-ineq} (\cite{BFS-JDDE2020}, Corollary 1.43)
  Suppose that $h:[a,\infty)\to (0,\infty)$ is a nondecreasing left continuous function.
  If $u:[a,\infty)\to (0,\infty)$ is bounded and Perron-Stieltjes integrable with respect to $h$ such that
\[
u(t)\leq c_1+c_2 \int_a^t u(s)dh(s), \quad t\in [a,\infty),
\]
where $c_1,c_2>0$ are constants, then
\[
u(t)\leq c_1 e^{c_2|h(t)-h(0)|}.
\]
\end{lemma}

%Then based on Lemma \ref{bellman-ineq}, we have the following result.
\begin{lemma}\label{xx-ineq}
  Suppose that the linear GODE \eqref{LLL} admits a strong exponential dichotomy and conditions (A1)-(A3) hold.
  Then we have
\[
\|X(t,0,x)-X(t,0,\widetilde{x})\| \leq K\|x-\widetilde{x}\|e^{(1+\mathscr{L})|h(t)-h(0)|}e^{\widetilde{\alpha}|t|}, \quad t\in\mathbb{R},
\]
where $\mathscr{L}:=KC^3 e^{3CV_A}V_A^2$.% and $X(t,0,x)$ is the solution of the nonlinear GODE \eqref{NL} with the initial value $X(0)=x$.
\end{lemma}

\begin{proof}
  Since $X(t,0,x)$ is the solution of the nonlinear GODE \eqref{NL}, we obtain that for any $x,\widetilde{x}\in \mathscr{X}$ and $t\geq 0$,
\[\begin{split}
 X(t,0,x)-X(t,0,\widetilde{x})=& \mathscr{V}(t)(x-\widetilde{x})+\int_0^t \left(DF(X(\tau,0,x),s)-DF(X(\tau,0,\widetilde{x}),s)\right) \\
 &-\int_0^t d_\sigma [\mathscr{V}(t)\mathscr{V}^{-1}(\sigma)]\int_0^\sigma \left(DF(X(\tau,0,x),s)-DF(X(\tau,0,\widetilde{x}),s)\right).
\end{split}\]
By taking the supremum norm  and using strong exponential dichotomy (see Definition \ref{s-ed}) and (A2), (A3), we have
\[\begin{split}
 \|X(t,0,x)-X(t,0,\widetilde{x})\|\leq& Ke^{\widetilde{\alpha}t}\|x-\widetilde{x}\|+\int_0^t \|X(s,0,x)-X(s,0,\widetilde{x})\|dh(s) \\
  &+\int_0^t \|d_\sigma [\mathscr{V}(t)\mathscr{V}^{-1}(\sigma)]\| \int_0^\sigma \|X(s,0,x)-X(s,0,\widetilde{x})\|dh(s).
\end{split}\]
  Set $Z(t):=e^{-\widetilde{\alpha}t} \|X(t,0,x)-X(t,0,\widetilde{x})\|$. Then
\[
Z(t)\leq K\|x-\widetilde{x}\|+\int_0^t e^{\widetilde{\alpha}(s-t)}Z(s)dh(s)+\left\|\int_0^td_\sigma [\mathscr{V}(t)\mathscr{V}^{-1}(\sigma)]\right\| \int_0^\sigma e^{\widetilde{\alpha}(s-t)}Z(s)dh(s).
\]
  We now deal with   $\left\|\int_0^td_\sigma [\mathscr{V}(t)\mathscr{V}^{-1}(\sigma)]\right\|$.
  Given a $\epsilon$-fine tagged division $D=\{(\tau_j,[s_{j-1},s_j]), j=1,2,\cdots,|D|\}$ of $[0,t]$, we deduce that
\[\begin{split}
 \left\|\int_0^td_\sigma [\mathscr{V}(t)\mathscr{V}^{-1}(\sigma)]\right\|=& \left\|\sum_{j=1}^{|D|}[\mathscr{V}(t)\mathscr{V}^{-1}(s_j)-\mathscr{V}(t)\mathscr{V}^{-1}(s_{j-1}) \right\| \\
 \leq& \sum_{j=1}^{|D|} \|\mathscr{V}(t)\|\|\mathscr{V}^{-1}(s_j)-\mathscr{V}^{-1}(s_{j-1})\|.
\end{split}\]
  By using  Theorem 6.15 in \cite{S-BOOK1992}, one has
\[
\|V(t,s)\|\leq Ce^{C \mathrm{var}_s^t A}, \quad \mathrm{var}_s^t V(\cdot,s)\leq Ce^{C \mathrm{var}_s^t A}\mathrm{var}_s^t A \quad \mathrm{and} \quad
\mathrm{var}_s^t V(t,\cdot)\leq C^2e^{2C\mathrm{var}_s^t A}\mathrm{var}_s^t A,
\]
  Recall that $\mathscr{V}^{-1}(t)=V(0,t)$.
  Then by condition (A2), the following estimate holds:
\[\begin{split}
\sum_{j=1}^{|D|} \|\mathscr{V}(t)\|\|\mathscr{V}^{-1}(s_j)-\mathscr{V}^{-1}(s_{j-1})\|\leq&
Ke^{\widetilde{\alpha}(t-s)} \|\mathscr{V}(s)\| C^2 e^{2C \mathrm{var}_s^t A}\mathrm{var}_s^t A \\
 \leq & Ke^{\widetilde{\alpha}(t-s)}C^3 e^{3CV_A}V_A^2.
\end{split}\]
 Therefore,
\[\begin{split}
 Z(t)\leq& K\|x-\widetilde{x}\|+\int_0^t e^{\widetilde{\alpha}(s-t)}Z(s)dh(s)+\int_0^\sigma KC^3 e^{3CV_A}V_A^2 \cdot Z(s) dh(s) \\
       \leq& K\|x-\widetilde{x}\|+\int_0^t Z(s)dh(s)+\int_0^t KC^3 e^{3CV_A}V_A^2 \cdot Z(s) dh(s) \\
       \leq& K\|x-\widetilde{x}\|+\int_0^t (1+KC^3 e^{3CV_A}V_A^2)\cdot Z(s)dh(s),
\end{split}\]
which implies that
\[Z(t)\leq K\|x-\widetilde{x}\|e^{(1+KC^3 e^{3CV_A}V_A^2)(h(t)-h(0))},\]
due to Lemma \ref{bellman-ineq}. Hence, for all $t\geq 0$, we conclude that
\[
\|X(t,0,x)-X(t,0,\widetilde{x})\| \leq K\|x-\widetilde{x}\|e^{(1+KC^3 e^{3CV_A}V_A^2)(h(t)-h(0))}e^{\widetilde{\alpha}t}.
\]
A similar conclusion is reached for $t\leq 0$.
%  Similarly, for $t\leq 0$, we have
%\[
%\|X(t,0,x)-X(t,0,\widetilde{x})\| \leq K\|x-\widetilde{x}\|e^{(1+KC^3 e^{3CV_A}V_A^2)(h(0)-h(t))}e^{-\widetilde{\alpha}t}.
%\]
\end{proof}

\subsection{H\"{o}lder continuous for the conjugacy $\Psi$}
  We claim in this subsection that $\Psi(t,x):=x+\psi(t,x)$ is H\"{o}lder continuous with respect to $x$.
  For this purpose, we estimate $\psi(t,x)-\psi(t,\widetilde{x})$ for any $t\in\mathbb{R}$ and $x,\widetilde{x}\in \mathscr{X}$.
  In fact,
\[\begin{split}
 \psi(t,x)-\psi(t,\widetilde{x})=& -\int_0^t (DF(X(r,t,x),s)-DF(X(r,t,x),s))\\
&+\int_{-\infty}^t d_\sigma[\mathscr{V}(t)P\mathscr{V}^{-1}(\sigma)]\int_0^\sigma (DF(X(r,t,x),s)-DF(X(r,t,x),s))\\
&-\int_t^\infty d_\sigma[\mathscr{V}(t)(I-P)\mathscr{V}^{-1}(\sigma)]\int_0^\sigma (DF(X(r,t,x),s)-DF(X(r,t,x),s))\\
:=& \mathcal{I}_1+ \mathcal{I}_2+ \mathcal{I}_3.
\end{split}\]
  We then deal with $\mathcal{I}_1$, $\mathcal{I}_2$ and $\mathcal{I}_3$.
  Without loss of generality, assume that $0<\|x-\widetilde{x}\|<1$.
  Set $\theta=\frac{1}{\widetilde{\alpha}+\alpha}\ln\frac{1}{\|x-\widetilde{x}\|}$ such that $t-\theta\geq \theta$, where $\widetilde{\alpha}$ is given by Definition \ref{s-ed}.
  Then we split $\mathcal{I}_1$, $\mathcal{I}_2$ and $\mathcal{I}_3$ into several parts:
\[\begin{split}
\mathcal{I}_1=& -\int_0^\theta-\int_\theta^{t-\theta}-\int_{t-\theta}^t=\mathcal{I}_{11}+\mathcal{I}_{12}+\mathcal{I}_{13},\\
\mathcal{I}_2=& \int_{-\infty}^{t-\theta}+\int_{t-\theta}^t=\mathcal{I}_{21}+\mathcal{I}_{22},\\
\mathcal{I}_3=& -\int_{t+\theta}^{\infty}-\int^{t+\theta}_t=\mathcal{I}_{31}+\mathcal{I}_{32}.
\end{split}\]
  From the definition of Kurzweil integral and conditions (A3), \eqref{SFF}, there is a $\epsilon$-fine tagged division of $[0,\theta]$ such that
\[\begin{split}
\|\mathcal{I}_{11}\|\leq&  2\left\|\int_0^{\theta} DF(X(r,t,x),s) \right\| \leq 2\sum_{j=1}^{|D|} e^{-\alpha(s_j-s_{j-1})}|h(s_j)-h(s_{j-1})| \\
\leq& 4V_h e^{-\alpha\theta}\leq 4V_h \|x-\widetilde{x}\|^{\frac{\alpha}{\widetilde{\alpha}+\alpha}}.
\end{split}\]
 Similarly, we also have
\[
\|\mathcal{I}_{12}\|\leq   2\left\|\int_\theta^{t-\theta} DF(X(r,t,x),s) \right\|
%\leq 2\left\|\int_0^{t-\theta} DF(X(r,t,x),s) \right\|  \\
%\leq& 4V_h e^{-\alpha (t-\theta)}\leq 4V_h e^{-\alpha\theta}
\leq 4V_h \|x-\widetilde{x}\|^{\frac{\alpha}{\widetilde{\alpha}+\alpha}}.
 \]
  By using conditions (A1)--(A3) and \eqref{aaa}, \eqref{bbb}, we deduce that
\[\begin{split}
\|\mathcal{I}_{21}\|\leq& 2\left\|\int_{-\infty}^{t-\theta}d_\sigma[\mathscr{V}(t)P\mathscr{V}^{-1}(\sigma)]\int_0^\sigma DF(X(r,t,x),s) \right\| \\
\leq& 4 e^{-\alpha \sigma}V_h \lim\limits_{\eta\to\infty} \left\| \int_{-\eta}^{t-\theta} d_\sigma[\mathscr{V}(t)P\mathscr{V}^{-1}(\sigma)] \right\| \\
\leq& 4V_h e^{-\alpha\sigma} K^2 C^3 e^{3CV_A}V_A^2 \lim\limits_{\eta\to\infty} e^{-\alpha(t-\theta+\eta)}\\
\leq& 4V_h e^{-\alpha\sigma} K^2 C^3 e^{3CV_A}V_A^2 e^{-\alpha\theta}\\
\leq& 4V_h e^{-\alpha\sigma} K^2 C^3 e^{3CV_A}V_A^2 \|x-\widetilde{x}\|^{\frac{\alpha}{\widetilde{\alpha}+\alpha}}.
\end{split}\]
  By a similar procedure, we have
\[\begin{split}
\|\mathcal{I}_{31}\|\leq& 2\left\|\int_{t+\theta}^{\infty}d_\sigma[\mathscr{V}(t)(I-P)\mathscr{V}^{-1}(\sigma)]\int_0^\sigma DF(X(r,t,x),s) \right\| \\
\leq& 4V_h e^{-\alpha\sigma} K(1+K) C^3 e^{3CV_A}V_A^2 \|x-\widetilde{x}\|^{\frac{\alpha}{\widetilde{\alpha}+\alpha}}.
\end{split}\]
 It remains to estimate $\mathcal{I}_{13}$,  $\mathcal{I}_{22}$ and $\mathcal{I}_{32}$.
  From condition (A3) and Lemma \ref{xx-ineq}, we have
\[\begin{split}
 \|\mathcal{I}_{13}\|\leq& \int_{t-\theta}^t \|X(s,t,x)-X(s,t,\widetilde{x})\|dh(s)  \\
 \leq& \int_{t-\theta}^t K\|x-\widetilde{x}\| e^{2(1+\mathscr{L})V_h}e^{\widetilde{\alpha}|s-t|}dh(s)\\
 \leq& \sum_{j=1}^{|D|} K\|x-\widetilde{x}\| e^{2(1+\mathscr{L})V_h} e^{\widetilde{\alpha}(t-s_j)}|h(s_j)-h(s_{j-1})| \\
 \leq& 2Ke^{2(1+\mathscr{L})V_h}e^{\widetilde{\alpha}\theta}\theta V_h \|x-\widetilde{x}\| \\
 \leq& 2K\theta V_h e^{2(1+\mathscr{L})V_h}\|x-\widetilde{x}\|^{\frac{\alpha}{\widetilde{\alpha}+\alpha}},
\end{split}\]
\[\begin{split}
\|\mathcal{I}_{22}\|\leq& \left\|\int_{t-\theta}^t d_\sigma[\mathscr{V}(t)P\mathscr{V}^{-1}(\sigma)]\right\| \int_0^\sigma \|X(s,t,x)-X(s,t,\widetilde{x})\|dh(s)  \\
\leq& 2K\theta V_h e^{2(1+\mathscr{L})V_h}K^2 C^3 e^{3CV_A}V_A^2 e^{(\widetilde{\alpha}-\alpha)\theta}\|x-\widetilde{x}\| \\
\leq& 2\theta K^3 C^3 e^{3CV_A}V_A^2 V_h e^{2(1+\mathscr{L})V_h}\|x-\widetilde{x}\|^{\frac{2\alpha}{\widetilde{\alpha}+\alpha}}
\end{split}\]
and
\[\begin{split}
\|\mathcal{I}_{32}\|\leq& \left\|\int_{t-\theta}^t d_\sigma[\mathscr{V}(t)(I-P)\mathscr{V}^{-1}(\sigma)]\right\| \int_0^\sigma \|X(s,t,x)-X(s,t,\widetilde{x})\|dh(s)  \\
\leq& 2\theta (1+K)K^2 C^3 e^{3CV_A}V_A^2 V_h e^{2(1+\mathscr{L})V_h}\|x-\widetilde{x}\|^{\frac{2\alpha}{\widetilde{\alpha}+\alpha}}.
\end{split}\]
  Hence, it follows from the above all inequalities that
\[\begin{split}
 \|\psi(t,x)-\psi(t,\widetilde{x})\|\leq& \|\mathcal{I}_1\|+\|\mathcal{I}_2\|+\|\mathcal{I}_3\| \\
\leq& 8V_h \|x-\widetilde{x}\|^{\frac{\alpha}{\widetilde{\alpha}+\alpha}}+8V_h e^{-\alpha\sigma} K^2 C^3 e^{3CV_A}V_A^2 \|x-\widetilde{x}\|^{\frac{\alpha}{\widetilde{\alpha}+\alpha}} \\
&+ 2K\theta V_h e^{2(1+\mathscr{L})V_h}\|x-\widetilde{x}\|^{\frac{\alpha}{\widetilde{\alpha}+\alpha}}
+2\theta K^3 C^3 e^{3CV_A}V_A^2 V_h e^{2(1+\mathscr{L})V_h}\|x-\widetilde{x}\|^{\frac{2\alpha}{\widetilde{\alpha}+\alpha}} \\
&+2\theta (1+K)K^2 C^3 e^{3CV_A}V_A^2 V_h e^{2(1+\mathscr{L})V_h}\|x-\widetilde{x}\|^{\frac{2\alpha}{\widetilde{\alpha}+\alpha}} \\
\leq& C_1 \|x-\widetilde{x}\|^{\frac{\alpha}{\widetilde{\alpha}+\alpha}},
\end{split}\]
for some constant $C_1>0$. Consequently, if $\|x-\widetilde{x}\|<1$, then
\[
\|\Psi(t,x)-\Psi(t,\widetilde{x})\|\leq \|x-\widetilde{x}\|+\|\psi(t,x)-\psi(t,\widetilde{x})\|\leq C \|x-\widetilde{x}\|^{\frac{\alpha}{\widetilde{\alpha}+\alpha}},
\]
for some constant $C>0$.

\subsection{H\"{o}lder continuous for the conjugacy $\Phi$}
  Next, we claim that $\Phi(t,x):=x+\phi(t,x)$ is H\"{o}lder continuous with respect to $x$.
 By Step 1 in Theorem \ref{linearization-thm}, we see that
\[
\lim\limits_{k\to\infty}\phi_k(t,x)= \phi(t,x), \quad  k\in\mathbb{N}.
\]
%  is the fixed point of the contraction map $\mathscr{T}$
%\[\begin{split}
% (\mathscr{T}\phi)(t,x)=&\int_0^t DF(Y(r,t,x)+\phi(r,Y(r,t,x)),s) \\
% &-\int_{-\infty}^t d_\sigma [\mathscr{V}(t)P\mathscr{V}^{-1}(\sigma)]\int_0^\sigma DF(Y(r,t,x)+\phi(r,Y(r,t,x)),s) \\
% &+\int_t^\infty d_\sigma [\mathscr{V}(t)(I-P)\mathscr{V}^{-1}(\sigma)]\int_0^\sigma DF(Y(r,t,x)+\phi(r,Y(r,t,x)),s),
%\end{split}\]
%where
Furthermore,
  we define $\phi_0(t,x)=0$  for any $t\in\mathbb{R}$ and $x\in \mathscr{X}$, and by recursion define
\[\begin{split}
\phi_{k+1}(t,x)=&\int_0^t DF(Y(r,t,x)+\phi_k(r,Y(r,t,x)),s) \\
 &-\int_{-\infty}^t d_\sigma [\mathscr{V}(t)P\mathscr{V}^{-1}(\sigma)]\int_0^\sigma DF(Y(r,t,x)+\phi_k(r,Y(r,t,x)),s) \\
 &+\int_t^\infty d_\sigma [\mathscr{V}(t)(I-P)\mathscr{V}^{-1}(\sigma)]\int_0^\sigma DF(Y(r,t,x)+\phi_k(r,Y(r,t,x)),s).
\end{split}\]
%for $k\in\mathbb{N}$. Clearly,
%\[
%\phi_k(t,x)\to \phi(t,x), \quad \mathrm{as} \; k\to \infty.
%\]
  Now we prove that if $0<\|x-\widetilde{x}\|<1$, then
\begin{equation}\label{Lgg}
  \|\phi_k(t,x)-\phi_k(t,\widetilde{x})\|\leq \widetilde{C}\|x-\widetilde{x}\|^q,
\end{equation}
where $\widetilde{C}>0$ and $0<q\leq \frac{\alpha}{\widetilde{\alpha}+\alpha}$.
   If $k=0$, then the inequality \eqref{Lgg} holds obviously. Making the inductive assumption that \eqref{Lgg} holds, we have
\[\begin{split}
\|\phi_{k+1}(t,x)-\phi_{k+1}(t,\widetilde{x})\|
\leq & \left\|\int_0^t \hbar(r) \right\|
 +\left\|\int_{-\infty}^t d_\sigma [\mathscr{V}(t)P\mathscr{V}^{-1}(\sigma)]\int_0^\sigma \hbar(r)\right\|  \\
 &+\left\|\int_t^\infty d_\sigma [\mathscr{V}(t)(I-P)\mathscr{V}^{-1}(\sigma)]\int_0^\sigma \hbar(r)\right\| \\
 :=& \mathcal{J}_1+\mathcal{J}_2+\mathcal{J}_3,
\end{split}\]
where
\[\hbar(r)=DF(Y(r,t,x)+\phi_k(r,Y(r,t,x)),s)-DF(Y(r,t,\widetilde{x})+\phi_k(r,Y(r,t,\widetilde{x})),s).\]
  We also split $\mathcal{J}_1$, $\mathcal{J}_2$ and  $\mathcal{J}_3$ into several parts:
\[\begin{split}
\mathcal{J}_1=& \int_0^\theta+\int_\theta^{t-\theta}+\int_{t-\theta}^t=\mathcal{J}_{11}+\mathcal{J}_{12}+\mathcal{J}_{13},\\
\mathcal{J}_2=& -\int_{-\infty}^{t-\theta}-\int_{t-\theta}^t=\mathcal{J}_{21}+\mathcal{J}_{22},\\
\mathcal{J}_3=& \int_{t+\theta}^{\infty}+\int^{t+\theta}_t=\mathcal{J}_{31}+\mathcal{J}_{32},
\end{split}\]
where $\theta=\frac{1}{\widetilde{\alpha}+\alpha}\ln \frac{1}{\|x-\widetilde{x}\|}$ such that $t-\theta\geq \theta$.
  Similar to $\mathcal{I}_{11}$, $\mathcal{I}_{12}$,  $\mathcal{I}_{21}$ and $\mathcal{I}_{31}$,  the following estimates are valid:
\[\begin{split}
 \|\mathcal{J}_{11}\|\leq&  2\left\|\int_0^{\theta} DF(Y(r,t,x)+\phi_k(r,Y(r,t,x)),s) \right\| \\
\leq& 2\sum_{j=1}^{|D|} e^{-\alpha(s_j-s_{j-1})}|h(s_j)-h(s_{j-1})|
\leq 4V_h \|x-\widetilde{x}\|^{\frac{\alpha}{\widetilde{\alpha}+\alpha}},\\
 \|\mathcal{J}_{12}\|\leq&  2\left\|\int_\theta^{t-\theta} DF(Y(r,t,x)+\phi_k(r,Y(r,t,x)),s) \right\|
%\leq& 2\sum_{j=1}^{|D|} e^{-\alpha(s_j-s_{j-1})}|h(s_j)-h(s_{j-1})| \\
\leq4V_h \|x-\widetilde{x}\|^{\frac{\alpha}{\widetilde{\alpha}+\alpha}},\\
\|\mathcal{J}_{21}\|\leq& 2\left\|\int_{-\infty}^{t-\theta}d_\sigma[\mathscr{V}(t)P\mathscr{V}^{-1}(\sigma)]\int_0^\sigma DF(Y(r,t,x)+\phi_k(r,Y(r,t,x)),s)  \right\| \\
\leq& 4V_h e^{-\alpha\sigma} K^2 C^3 e^{3CV_A}V_A^2 \|x-\widetilde{x}\|^{\frac{\alpha}{\widetilde{\alpha}+\alpha}},\\
\|\mathcal{J}_{31}\|\leq& 2\left\|\int_{t+\theta}^{\infty}d_\sigma[\mathscr{V}(t)(I-P)\mathscr{V}^{-1}(\sigma)]\int_0^\sigma DF(Y(r,t,x)+\phi_k(r,Y(r,t,x)),s)  \right\| \\
\leq& 4V_h e^{-\alpha\sigma} (1+K)K C^3 e^{3CV_A}V_A^2 \|x-\widetilde{x}\|^{\frac{\alpha}{\widetilde{\alpha}+\alpha}}.
\end{split}\]
  Notice that
\[
\|Y(s,t,x)-Y(s,t,\widetilde{x})\|\leq Ke^{\widetilde{\alpha}|t-s|}\|x-\widetilde{x}\|
\]
and by \eqref{Lgg},
\[
\|\phi_k(s,Y(s,t,x))-\phi_k(s,Y(s,t,\widetilde{x}))\|\leq \widetilde{C}K^q e^{\widetilde{\alpha}q|t-s|}\|x-\widetilde{x}\|^q,
\]
for $t,s\in\mathbb{R}$ and $x,\widetilde{x}\in \mathscr{X}$. Then we have
\[\begin{split}
 \|\mathcal{J}_{13}\|\leq& \int_{t-\theta}^t \|Y(s,t,x)-Y(s,t,\widetilde{x})+\phi_k(s,Y(s,t,x))-\phi_k(s,Y(s,t,\widetilde{x}))\|dh(s) \\
 \leq& \int_{t-\theta}^t Ke^{\widetilde{\alpha}(t-s)}\|x-\widetilde{x}\|dh(s)
                +\int_{t-\theta}^t \widetilde{C}K^q e^{\widetilde{\alpha}q(t-s)}\|x-\widetilde{x}\|^q dh(s) \\
 \leq& K \|x-\widetilde{x}\|\sum_{j=1}^{|D|} e^{\widetilde{\alpha}(t-s_j)}|h(s_j)-h(s_{j-1})|
                +\widetilde{C}K^q \|x-\widetilde{x}\|^{q} \sum_{j=1}^{|D|} e^{\widetilde{\alpha}q(t-s_j)}|h(s_j)-h(s_{j-1})| \\
 \leq& 2K\theta V_h \|x-\widetilde{x}\|^{\frac{\widetilde{\alpha}}{\widetilde{\alpha}+\alpha}}
                +2\widetilde{C}K^q \theta V_h  \|x-\widetilde{x}\|^{\frac{\widetilde{\alpha}q}{\widetilde{\alpha}+\alpha}},
\end{split}\]
and
\begin{small}
\[\begin{split}
 \|\mathcal{J}_{22}\|\leq& \left\|\int_{t-\theta}^t d_\sigma[\mathscr{V}(t)P\mathscr{V}^{-1}(\sigma)]\right\|
        \int_0^\sigma \|Y(s,t,x)-Y(s,t,\widetilde{x})+\phi_k(s,Y(s,t,x))-\phi_k(s,Y(s,t,\widetilde{x}))\|dh(s) \\
 \leq& 2K^3\theta C^3 e^{3CV_A}V_h V_A^2 e^{(\widetilde{\alpha}-\alpha)\theta}\|x-\widetilde{x}\|
                +2\widetilde{C}K^{2+q} \theta C^3 e^{3CV_A}V_h V_A^2 e^{(\widetilde{\alpha}q-\alpha)\theta}\|x-\widetilde{x}\|^q \\
 \leq& 2K^3\theta C^3 e^{3CV_A}V_h V_A^2  \|x-\widetilde{x}\|^{\frac{2\alpha}{\widetilde{\alpha}+\alpha}}
             +2\widetilde{C}K^{2+q} \theta C^3 \|x-\widetilde{x}\|^{\frac{\alpha q+\alpha}{\widetilde{\alpha}+\alpha}},
\end{split}\]
\end{small}
and
\begin{small}
\[\begin{split}
 \|\mathcal{J}_{32}\|\leq& \left\|\int_{t}^{t+\theta} d_\sigma[\mathscr{V}(t)(I-P)\mathscr{V}^{-1}(\sigma)]\right\|
        \int_0^\sigma \|Y(s,t,x)-Y(s,t,\widetilde{x})+\phi_k(s,Y(s,t,x))-\phi_k(s,Y(s,t,\widetilde{x}))\|dh(s) \\
 \leq& 2(1+K)K^2\theta C^3 e^{3CV_A}V_h V_A^2  \|x-\widetilde{x}\|^{\frac{2\alpha}{\widetilde{\alpha}+\alpha}}
             +2\widetilde{C}(1+K)K^{1+q} \theta C^3 \|x-\widetilde{x}\|^{\frac{\alpha q+\alpha}{\widetilde{\alpha}+\alpha}}.
\end{split}\]
\end{small}
  Since $0<q\leq \frac{\alpha}{\widetilde{\alpha}+\alpha}$, and from the above all inequalities, we conclude that
\[\begin{split}
\|\phi_{k+1}(t,x)-\phi_{k+1}(t,\widetilde{x})\|\leq& \|\mathcal{J}_1\|+ \|\mathcal{J}_2\|+ \|\mathcal{J}_3\| \\
  \leq& 8V_h \|x-\widetilde{x}\|^{\frac{\alpha}{\widetilde{\alpha}+\alpha}}+4V_h e^{-\alpha\sigma} K^2 C^3 e^{3CV_A}V_A^2 \|x-\widetilde{x}\|^{\frac{\alpha}{\widetilde{\alpha}+\alpha}}\\
  &+4V_h e^{-\alpha\sigma} (1+K)K C^3 e^{3CV_A}V_A^2 \|x-\widetilde{x}\|^{\frac{\alpha}{\widetilde{\alpha}+\alpha}}
  +2K\theta V_h \|x-\widetilde{x}\|^{\frac{\widetilde{\alpha}}{\widetilde{\alpha}+\alpha}} \\
  &+2\widetilde{C}K^q \theta V_h  \|x-\widetilde{x}\|^{\frac{\widetilde{\alpha}q}{\widetilde{\alpha}+\alpha}}
   +2K^3\theta C^3 e^{3CV_A}V_h V_A^2  \|x-\widetilde{x}\|^{\frac{2\alpha}{\widetilde{\alpha}+\alpha}}\\
   &+2\widetilde{C}K^{2+q} \theta C^3 \|x-\widetilde{x}\|^{\frac{\alpha q+\alpha}{\widetilde{\alpha}+\alpha}}
   +2(1+K)K^2\theta C^3 e^{3CV_A}V_h V_A^2  \|x-\widetilde{x}\|^{\frac{2\alpha}{\widetilde{\alpha}+\alpha}}\\
   &+2\widetilde{C}(1+K)K^{1+q} \theta C^3 \|x-\widetilde{x}\|^{\frac{\alpha q+\alpha}{\widetilde{\alpha}+\alpha}} \\
\leq& \widetilde{C}\|x-\widetilde{x}\|^q,
\end{split}\]
for some constant $\widetilde{C}>0$, which implies that \eqref{Lgg} holds for all $k$.
  Consequently, if $\|x-\widetilde{x}\|<1$, then
\[
\|\Phi(t,x)-\Phi(t,\widetilde{x})\| \leq \widehat{C}\|x-\widetilde{x}\|^q.
\]

\section{Applications}

\subsection{A Hartman-Grobman theorem for Measure differential equations (MDEs)}
\subsubsection{Fundamental theory for MDEs}
Let $\mathscr{X}$  be a Banach space and  $\mathbb{I}\subset \mathbb{R}$ be an interval.
Consider the linear MDE
\begin{equation}\label{ML}
  Dx=\mathscr{A}(t)x+\mathscr{C}(t)Du,
\end{equation}
where $Dx$ and $Du$ denote the distributional derivatives of $x$ and $u$, and the functions
$\mathscr{A}: \mathbb{I}\to \mathscr{B}(\mathscr{X})$, $\mathscr{C}: \mathbb{I}\to \mathscr{B}(\mathscr{X})$ and $u:\mathbb{I}\to\mathbb{R}$
satisfying the following conditions:
\\
(D1) $\mathscr{A}(t)$ is Perron integrable for any  $t\in \mathbb{I}$.\\
(D2) $u(t)$ is of locally bounded variation for any  $t\in \mathbb{I}$ and continuous from the left on $\mathbb{I}\backslash \{\inf \mathbb{I}\}$. \\
(D3) $du$ denotes the Lebesgue-Stieltjes measure generated by the
function $u$, $\mathscr{C}(t)$ is Perron-Stieltjes integrable in  $u$ for any  $t\in \mathbb{I}$.\\
Moreover, we consider some additional assumptions:\\
(D4) There is a Lebesgue measure function $m_1:\mathbb{I}\to\mathbb{R}$ satisfying for any $c,d\in\mathbb{I}$, we have
$\int_c^d m_1(s)ds<\infty$ and
\[
\left\|\int_c^d \mathscr{A}(s)ds\right\|\leq \int_c^d m_1(s)ds.
\]
(D5) There exists a function $du$-measurable $m_2:\mathbb{I}\to\mathbb{R}$ satisfying for any $c,d\in\mathbb{I}$, we have
$\int_c^d m_2(s)du(s)<\infty$ and
\[
\left\|\int_c^d \mathscr{C}(s)du(s)\right\|\leq \int_c^d m_2(s)du(s).
\]
(D6) For all $t$ such that $t$ is a point of discontinuity of $u$, we have
\[
\left(I+\lim\limits_{r\to t^+}\int_t^r\mathscr{C}(s)du(s)\right)^{-1}\in \mathscr{B}(\mathscr{X}).
\]

By  (D1)--(D3), one says that $x:[c,d]\subset \mathbb{I}\to \mathscr{X}$ is a solution of \eqref{ML} satisfying the initial value $x(t_0)=x_0$, if
\[
x(t)=x_0+\int_{t_0}^{t} \mathscr{A}(s)x(s)ds+\int_{t_0}^t \mathscr{C}(s)x(s)du(s).
\]
   If all conditions (D1)--(D6) hold, then the existence and uniqueness of a solution of \eqref{ML} associated to the initial value $x(t_0)=x_0$ follows
from Theorem 5.2 in \cite{BFS-JDE2018} immediately.
  Hence, conditions (D1)--(D6) are valid throughout this subsection.

\begin{lemma}(\cite{S-BOOK1992} Theorem 5.17)
 Given $t_0\in [c,d]$, %For an initial condition $x(t_0)=x_0$,
  the function $x:[c,d]\subset\mathbb{I}\to \mathscr{X}$ is a solution of \eqref{ML} iff  $x$
is a solution of
\begin{equation}\label{GGL}
  \begin{cases}
  \frac{dx}{d\tau}=D[A(t)x+G(t)x],\\
  x(t_0)=x_0,
  \end{cases}
\end{equation}
where $A(t)=\int_{t_0}^{t} \mathscr{A}(s)ds$ and $G(t)=\int_{t_0}^t \mathscr{C}(s)du(s)$.
\end{lemma}

   The fundamental operator $U:\mathbb{I}\times \mathbb{I}\to \mathscr{B}(\mathscr{X})$ of MDEs \eqref{ML} was given by \cite{BFS-JDE2018} satisfying
\begin{equation}\label{MMU}
  U(t,s)=I+\int_s^t \mathscr{A}(r)U(r,s)dr+\int_s^t \mathscr{C}(r)U(r,s)du(r), \quad t,s\in\mathbb{I}.
\end{equation}
  Moreover,
  %for each fixed $s\in\mathbb{I}$, $U(\cdot,s)$ is an operator of locally bounded variation in $\mathbb{I}$,
the function $x(t)=U(t,t_0)x_0$ is the solution of
\eqref{ML} satisfying the initial value $x(t_0)=x_0\in \mathscr{X}$.

%  The exponential  for the MDE \eqref{ML} was introduced in.
 % Set $\mathscr{U}(t)=U(t,t_0)$ and $\mathscr{U}^{-1}(t)=U(t_0,t)$.
\begin{definition}\label{MED}(exponential dichotomy, \cite{BFS-JDE2018})
The MDEs \eqref{ML} have an exponential dichotomy with $(P,K,\alpha)$ on $\mathbb{I}$,  if there exist  a projection $P:\mathscr{X} \to \mathscr{X}$ and
constants $K, \alpha$  satisfying
\begin{equation}\label{M1}
\begin{cases}
 \|\mathscr{U}(t) P\mathscr{U}^{-1}(s)\|\leq K e^{-\alpha(t-s)}, \quad t\geq s;\\
 \|\mathscr{U}(t)(I-P)\mathscr{U}^{-1}(s)\|\leq Ke^{\alpha(t-s)}, \quad t\leq s,
\end{cases}
\end{equation}
where $\mathscr{U}(t)=U(t,0)$ and $\mathscr{U}^{-1}(t)=U(0,t)$.
\end{definition}

\begin{definition}(strong exponential dichotomy)
  The MDEs \eqref{ML} have a strong exponential dichotomy on $\mathbb{I}$,  if \eqref{M1} holds and there exists a positive constant $\widetilde{\alpha}>\alpha$
such that
\begin{equation}\label{M2}
   \|\mathscr{U}(t) \mathscr{U}^{-1}(s)\|\leq Ke^{\widetilde{\alpha}|t-s|}, \quad \mathrm{for} \; t,s\in\mathbb{R}.
\end{equation}
\end{definition}

\begin{lemma}\label{lemma-ed}(\cite{BFS-JDE2018} Proposition 5.7)
The MDEs \eqref{ML} have an exponential dichotomy with $(P,K,\alpha)$ iff the GODEs
\begin{equation}\label{GGODE}
  \frac{dx}{d\tau}=D[A(t)x+F(t)x], \quad t\in\mathbb{I},
\end{equation}
have an exponential dichotomy with $(P,K,\alpha)$,
where $A(t)=\int_{t_0}^{t} \mathscr{A}(s)ds$ and $G(t)=\int_{t_0}^t \mathscr{C}(s)du(s)$.
\end{lemma}

\subsubsection{Main results for MDEs}

Consider the nonlinear MDEs as follows.
\begin{equation}\label{MNL}
  Dx=\mathscr{A}(t)x+\mathscr{C}(t)xDu+\mathscr{H}(t,x)Du,
\end{equation}
where $\mathscr{H}(t,x):\mathbb{R} \times \mathscr{X} \to \mathscr{X}$ is  Lebesgue-Stieltjes integrable with respect to $u$ and $\mathscr{H}(t,0)=0$.
Suppose that:\\
(a)  for all $t$ such that $t$ is a point of discontinuity of $u$, there exists a positive constant $C_g>0$ satisfying
\[
\left\|\left(Id+\lim\limits_{r\to t^+}\int_t^r\mathscr{C}(s)du(s)\right)^{-1}\right\|\leq C_g;
\]
(b)  $u$ is a bounded variation function on $\mathbb{R}$ and $u$ is nondecreasing, i.e.,
\[
V_u:=\sup\{\mathrm{var}_c^d u: c,d \in\mathbb{R}, c<d\}<\infty;
\]
(c)
  $\mathscr{H}(t,x)$ is uniformly bounded in $t\in\mathbb{R}$ with constant $M_h>0$ for any $x\in \mathscr{X}$, i.e.,
\[
\|\mathscr{H}(t,x)\|\leq M_h;
\]
(d) there is a sufficiently small Lipschitz constant $L_h$ such that for any $t\in\mathbb{R}$ and $x,\widetilde{x}\in \mathscr{X}$,
\[
\|\mathscr{H}(t,x)-\mathscr{H}(t,\widetilde{x})\|\leq L_h \|x-\widetilde{x}\|.
\]

  Then we establish our main results for MDEs.
\begin{theorem}\label{mde-thm1}
  Suppose that linear MDEs \eqref{ML} possess an exponential dichotomy with the form \eqref{M1}.
Further assume that conditions (a), (b), (c), (d) hold. If
\[
2 L_h V_u \left(1+K(1+2K)C_g^3e^{3C_g V_{A+F}}V_{A+F}^2\right)<1,
\]
then Eq. \eqref{MNL} has a unique bounded solution.
%, which has the form
%\[\begin{split}
%x(t)=& \int_{t_0}^t \mathscr{H}(s,x)du(s)-\int_{-\infty}^t d_s[\mathscr{U}(t)P\mathscr{U}^{-1}(s)]\int_{t_0}^s \mathscr{H}(s,x)du(s) \\
%&+\int_{t}^\infty d_s[\mathscr{U}(t)(I-P)\mathscr{U}^{-1}(s)]\int_{t_0}^s \mathscr{H}(s,x)du(s).
%\end{split}\]
\end{theorem}

\begin{theorem}\label{mde-thm2}
  Suppose that all conditions of Theorem \ref{mde-thm1} hold. Then the linear MDEs \eqref{ML} are topologically conjugated to the nonlinear MDEs \eqref{MNL}.
\end{theorem}

\begin{theorem}\label{mde-thm3}
 Suppose that linear MDEs \eqref{ML} admits a strong exponential dichotomy with the form \eqref{M2} and all conditions of Theorem \ref{mde-thm1} hold.
 If further the bounded variation function $u$ satisfies $|u(t)-u(s)|\leq e^{-\alpha|t-s|}$, then the conjugacies are both H\"{o}lder continuous.
\end{theorem}

We only prove Theorem \ref{mde-thm1}, and verify that Theorem  \ref{mde-thm1} satisfies all the conditions of Theorem \ref{Bound1}.
The proofs of the other two theorems are similar to Theorem \ref{mde-thm1}.

\begin{proof}[Proof of Theorem \ref{mde-thm1}.]
Since \eqref{ML} has an exponential dichotomy with $(P,K,\alpha)$, we derive from Lemma \ref{lemma-ed} that
\[
\frac{dx}{d\tau}=D[A(t)x+F(t)x]
\]
has an exponential dichotomy with the same  $(P,K,\alpha)$ on $\mathbb{R}$.

Note that the solution of Eq. \eqref{MNL} with the initial value $x(t_0)=x_0$ is defined by
\[
x(t)=x_0+\int_{t_0}^{t} \mathscr{A}(s)x(s) ds+\int_{t_0}^t \mathscr{C}(s)x(s)du(s)+\int_{t_0}^t \mathscr{H}(s,x(s))du(s).
\]
Given $t,t_0\in\mathbb{R}$ and set $A(t)=\int_{t_0}^{t}\mathscr{A}(s)ds$, $F(t)=\int_{t_0}^t \mathscr{C}(s)du(s)$.
We define the Kurzweil integrable map $\mathcal{N}:\mathscr{X} \times \mathbb{R}\to \mathscr{X}$
\begin{equation*}
  \mathcal{N}(x(t),t):=\int_{t_0}^t \mathscr{H}(s,x(s))du(s).
\end{equation*}
Then we have
\[
x(t)=x_0+\int_{t_0}^t d[A(s)]x(s)+\int_{t_0}^t d[F(s)]x(s)+\int_{t_0}^t D\mathcal{N}(x(s),s),
\]
which is a solution of the following nonlinear GODEs with the initial value $x(t_0)=x_0$
\begin{equation}\label{NGODE}
\frac{dx}{d\tau}=D[A(t)x+F(t)x+\mathcal{N}(x,t)].
\end{equation}
  We now claim that $\mathcal{N}(x,t)$ belongs to the class $\mathscr{A}(\Omega, u)$, here $\Omega=\mathscr{X}\times \mathbb{R}$. In fact,
for any $t,\widetilde{t}\in\mathbb{R}$ and $x\in \mathscr{X}$, by using condition (c), we have
\[
\|\mathcal{N}(x,t)-\mathcal{N}(x,\widetilde{t})\|= \left\| \int_{\widetilde{t}}^t \mathscr{H}(s,x)du(s)\right\|\leq \|\mathscr{H}(s,x)\| |u(t)-u(\widetilde{t})|\leq M_h |u(t)-u(\widetilde{t})|,
\]
and for any $t,\widetilde{t}\in\mathbb{R}$ and $x, \widetilde{x} \in \mathscr{X}$, by using conditions (c) and (d), we have
\[\begin{split}
 \|\mathcal{N}(x,t)-\mathcal{N}(x,\widetilde{t})-\mathcal{N}(\widetilde{x},t)+\mathcal{N}(\widetilde{x},\widetilde{t})\|= &
 \left\| \int_{\widetilde{t}}^t \left(\mathscr{H}(s,x)-\mathscr{H}(s,\widetilde{x})\right)du(s)\right\| \\
 \leq& \int_{\widetilde{t}}^t  \|\mathscr{H}(s,x)-\mathscr{H}(s,\widetilde{x})\|du(s) \\
 \leq&  \|\mathscr{H}(s,x)-\mathscr{H}(s,\widetilde{x})\| |u(t)-u(\widetilde{t})| \\
 \leq& L_h \|x-\widetilde{x}\| |u(t)-u(\widetilde{t})|.
\end{split}\]

    From Theorem 5.2 of \cite{BFS-JDE2018}, it follows that $\mathrm{var}_c^d (A+F)<\infty$ for all $c,d\in\mathbb{R}$ and $c<d$.
    For  convenience, we write $V_{A+F}=\mathrm{var}_c^d (A+F)$.
    Indeed, let $D=\{t_0,t_1,\cdots,t_{|D|}\}$ be a division of $[c,d]$. Then
\[
\sum_{j=1}^{|D|} \|A(t_j)+F(t_j)-A(t_{j-1})-F(t_{j-1})\|\leq
\sum_{j=1}^{|D|} \left\| \int_{t_{j-1}}^{t_j}\mathscr{A}(s)ds \right\|
+\sum_{j=1}^{|D|} \left\| \int_{t_{j-1}}^{t_j} \mathscr{C}(s)du(s) \right\|,
\]
by using condition (D4) and (D5), we deduce that
\[
\sum_{j=1}^{|D|} \left\| \int_{t_{j-1}}^{t_j}\mathscr{A}(s)ds \right\|
+\sum_{j=1}^{|D|} \left\| \int_{t_{j-1}}^{t_j} \mathscr{C}(s)du(s) \right\|
\leq \int_c^d m_1(s)ds+\int_c^d m_2(s)du(s) <\infty,
\]
that is, $V_{A+F}<\infty$.
Hence, taking $L_h$ is sufficiently small, and by using conditions (a), (b), (c) and (d), we can ensure that
\[2 L_h V_u (1+K(1+2K))C_g^3e^{3C_g V_{A+F}}V_{A+F}^2<1.\]
Consequently, all conditions of Theorem \ref{Bound1} are fulfilled, one concludes that Eq. \eqref{MNL} has a unique bounded solution.
\end{proof}

\subsection{A Hartman-Grobman theorem for the impulsive differential equations (IDEs)}
\subsubsection{Fundamental theory for IDEs}
 Denote that $\mathscr{X}$  is a Banach space and  $\mathbb{I}\subset \mathbb{R}$ is an interval.
 Consider the  linear IDEs
\begin{equation}\label{IL}
\begin{cases}
\dot{x}(t)=\widetilde{A}(t)x(t),\quad t\neq t_i, \\
\vartriangle x(t_i)=x(t_i^+)-x(t_i)=B_i x(t_i),\quad i\in\mathscr{I}:=\{i\in\mathbb{Z}:t_i\in \mathbb{I}\},
\end{cases}
\end{equation}
where $\widetilde{A}:\mathbb{I}\to\mathscr{B}(\mathscr{X})$ and $B_i\in \mathscr{B}(\mathscr{X})$ satisfy the following assumptions:\\
(B1) $\widetilde{A}(t)$ is Perron integrable for any $t\in \mathbb{I}$;\\
(B2) there is a Lebesgue measure function $m:\mathbb{I}\to\mathbb{R}$ satisfyinf for any $c,d\in \mathbb{I}$ and $c<d$, the Lebesgue integral $\int_c^d m(s)ds$ is finite and
\[
\left\|\int_c^d \widetilde{A}(s)ds \right\|\leq     \int_c^d m(s)ds.
\]
(B3) $(I+B_i)^{-1}\in \mathscr{B}(\mathscr{X})$, where $i\in\mathscr{I}$.\\
In addition, let $\{t_k\}_{k\in\mathbb{Z}}$ be the impulsive points satisfying the relation
\[
\cdots<t_{-k}<\cdots <t_{-1}<t_0=0<t_1<\cdots<t_k<\cdots,
\]
and $\lim\limits_{k\to \pm\infty}t_k=\pm\infty$. Set $\mathscr{I}_c^d:=\{i\in\mathscr{I}:c\leq t_i\leq d\}$ for $c,d\in \mathbb{I}$.
Define the Heaviside function $H_l$: % at point $l\in \mathbb{I}$
\[
H_l(t)=
\begin{cases}
0, \quad \mathrm{for} \; t\leq l,\\
1, \quad \mathrm{for} \; t>l.
\end{cases}
\]
  Then, the solution of Eq. \eqref{IL} with the initial value $x(t_0)=x_0$ satisfies % the following integral equation
\[
x(t)=\begin{cases}
 x_0+\int_{t_0}^t \widetilde{A}(s)x(s)ds+\sum\limits_{i\in\mathscr{I}_{t_0}^{t}}B_ix(t_i)H_{t_i}(t),\quad t\geq t_0 (t\in \mathbb{I}),\\
 x_0+\int_{t_0}^t \widetilde{A}(s)x(s)ds-\sum\limits_{i\in\mathscr{I}_{t}^{t_0}}B_ix(t_i)(1-H_{t_i}(t)),\quad t< t_0 (t\in \mathbb{I}).
\end{cases}
\]
Then by using Theorem 5.20 from \cite{S-BOOK1992}, $x(t)$ is a solution of Eq. \eqref{IL} iff $x(t)$ is a solution of the linear GODE $\frac{dx}{d\tau}=D[A(t)x]$, where $A$ is given by
%
%\begin{lemma} (\cite{S-BOOK1992} Theorem 5.20)
%Let $t_0\in \mathbb{I}$.  $x:\mathbb{I}\to \mathscr{X}$ is a solution of Eq. \eqref{IL} iff $x$
\begin{equation}\label{OA}
A(t)=\begin{cases}
 \int_{t_0}^t \widetilde{A}(s)ds+\sum\limits_{i\in\mathscr{I}_{t_0}^{t}}B_i H_{t_i}(t),\quad t\geq t_0,\\
 \int_{t_0}^t \widetilde{A}(s)ds-\sum\limits_{i\in\mathscr{I}_{t}^{t_0}}B_i(1-H_{t_i}(t)),\quad t< t_0.
\end{cases}
\end{equation}
%\end{lemma}

Let  $W:\mathbb{I}\times \mathbb{I}\to \mathscr{B}(\mathscr{X})$ be the evolution operator of the IDE \eqref{IL}, it has the following form:
if $t\geq s$, $t\in(t_i,t_{i+1}]$ and $s\in(t_{j-1},t_j]$, then
\[
W(t,s)=\Upsilon(t,t_k)\left( \prod_{k=i}^{j+1} [Id+B_k]\Upsilon(t_k,t_{k-1})\right)[Id+B_j]\Upsilon(t_j,s),
\]
where  $\Upsilon:\mathbb{I}\times \mathbb{I}\to \mathscr{B}(\mathscr{X})$ is the evolution operator of  $\dot{x}=\widetilde{A}(t)x$,
and if $t<s$, $s\in (t_j,t_{j+1}]$ and $t\in (t_{j-1},t_j]$, then
\[
W(t,s)=[W(s,t)]^{-1}=\Upsilon(t,t_j)[Id+B_j]^{-1}\cdot [Id+B_i]^{-1}\Upsilon(t_j,s).
\]
%if $t<s$, $s\in (t_j,t_{j+1}]$ and $t\in (t_{j-1},t_j]$ (with $j\leq i$ and $i,j\in\mathscr{I}$).
%
%Let $\Upsilon:\mathbb{I}\times \mathbb{I}\to \mathscr{B}(\mathscr{X})$ be the fundamental operator of the ODE $\dot{x}=\widetilde{A}(t)x$.
%Define an operator $W:\mathbb{I}\times \mathbb{I}\to \mathscr{B}(\mathscr{X})$ by
%
%
%Let $W(t,s), t,s\in \mathbb{I}$ be the fundamental operator of the IDE \eqref{IL} and set $\mathscr{W}(t)=W(t,t_0)$.
%Then the concept of exponential dichotomy is presented.
\begin{definition}\cite{BFS-JDE2018}
The IDEs \eqref{IL} possess an exponential dichotomy with $(P,K,\alpha)$ on $\mathbb{I}$ if there exist a projection $P$ and constants $K, \alpha$  such that
\begin{equation}\label{IED1}
\begin{cases}
 \|\mathscr{W}(t) P\mathscr{W}^{-1}(s)\|\leq K e^{-\alpha(t-s)}, \quad t\geq s,\\
 \|\mathscr{W}(t)(I-P)\mathscr{W}^{-1}(s)\|\leq Ke^{\alpha(t-s)}, \quad t< s.
\end{cases}
\end{equation}
where $\mathscr{W}(t)=W(t,t_0)$ and $\mathscr{W}^{-1}(t)=W(t_0,t)$.
\end{definition}

\begin{definition}
The IDEs \eqref{IL} admit a strong exponential dichotomy on $\mathbb{I}$,  if \eqref{IED1} holds and there exists a positive constant $\widetilde{\alpha}>\alpha$
such that
\begin{equation}\label{IED2}
   \|\mathscr{W}(t) \mathscr{W}^{-1}(s)\|\leq Ke^{\widetilde{\alpha}|t-s|}, \quad \mathrm{for} \; t,s\in\mathbb{R}.
\end{equation}
\end{definition}

\begin{lemma}\label{lemma-iied}
(\cite{BFS-JDE2018} Proposition 5.21)
The IDEs \eqref{IL} possess an exponential dichotomy with  $(P,K,\alpha)$ iff the GODEs
\begin{equation*}
  \frac{dx}{d\tau}=D[A(t)x], \quad t\in \mathbb{I},
\end{equation*}
have an exponential dichotomy with  $(P,K,\alpha)$,
where $A$ is defined by \eqref{OA}.
\end{lemma}

\subsubsection{Main results for IDEs}
Consider the nonlinear IDEs as follows:
\begin{equation}\label{INL}
\begin{cases}
\dot{x}(t)=\widetilde{A}(t)x(t)+f(t,x(t)), \quad t\neq t_i,\\
\vartriangle x(t_i)=B_i x(t_i),\quad i\in\mathbb{Z},
\end{cases}
\end{equation}
where $f:\mathbb{R}\times \mathscr{X}\to \mathscr{X}$ is Perron integrable and $f(t,0)=0$.
We suppose that the following conditions hold:\\
(a) for all $i\in \mathbb{Z}$, there exists a positive constant $C_b$ such that
\[
\sum\limits_{i\in\mathbb{Z}}\|B_i\|\leq C_b \quad \mathrm{and} \quad \|(Id+B_i)^{-1}\|\leq C_b;
\]
(b) there exists a Lebesgue measure function $m:\mathbb{R}\to\mathbb{R}$ such that the Lebesgue integral $\int_\mathbb{R} m(s)ds$ is finite and
\[
\left\|\int_{\mathbb{R}} \widetilde{A}(s)ds\right\|\leq \int_\mathbb{R} m(s)ds;
\]
(c) for any $t\in\mathbb{R}$ and $x,y\in \mathscr{X}$,
there exists a Lebesgue measure function $\gamma:\mathbb{R}\to\mathbb{R}$ such that the Lebesgue integral $\int_\mathbb{R} \gamma(s)ds$ is finite and
\[
\|f(t,x)\|\leq \gamma(t) \quad \mathrm{and} \quad     \| f(t,x)-f(t,y)\|\leq  \gamma(t)\|x-y\|.
\]
% the nonlinear perturbation $f(t,z)$ is uniformly bounded with some $M_f>0$ and is Lipschitzian
%in $z$ with a sufficiently small Lipschitz constant $L_f$.

 Then we give our main results for IDEs.
\begin{theorem}\label{ide-thm1}
Suppose that the IDEs \eqref{IL} possess an exponential dichotomy with   the form \eqref{IED1}. If conditions (a), (b), (c) hold and
$\int_\mathbb{R} \gamma(s)ds$ is sufficiently small,
then Eq. \eqref{INL} has a unique bounded solution.
\end{theorem}

\begin{theorem}\label{ide-thm2}
  Suppose that all conditions of Theorem \ref{ide-thm1} hold. Then Eq. \eqref{IL} is topologically conjugated to Eq. \eqref{INL}.
\end{theorem}

\begin{theorem}\label{ide-thm3}
 Suppose that linear IDEs \eqref{IL} admit a strong exponential dichotomy with the form \eqref{IED2} and all conditions of Theorem \ref{ide-thm1} hold.
 If further there exists a bounded nondecreasing function $\mu:\mathbb{R}\to\mathbb{R}$ such that  $|\mu(t)-\mu(s)|\leq e^{-\alpha|t-s|}$,
then the conjugacies are both H\"{o}lder continuous.
\end{theorem}

We only prove Theorem \ref{ide-thm1}, and verify that Theorem  \ref{ide-thm1} satisfies all the conditions of Theorem \ref{Bound1}.
The proofs of the other two theorems are similar to Theorem \ref{ide-thm1}.

\begin{proof}
Since IDEs \eqref{IL} possess an exponential dichotomy with $(P,K,\alpha)$, we derive from Lemma \ref{lemma-iied} that
\[
\frac{dx}{d\tau}=D[A(t)x]
\]
has an exponential dichotomy with the same  $(P,K,\alpha)$ on $\mathbb{R}$.

For any $t\geq t_0$, we note that the solution of Eq. \eqref{INL} with the initial value $x(t_0)=x_0$ is defined by
\[
x(t)=x_0+\int_{t_0}^{t} \widetilde{A}(s)x(s)ds+\sum\limits_{i\in\mathscr{I}_{t_0}^t}B_i x(t_i)H_{t_i}(t)+\int_{t_0}^t f(s,x(s))ds.
\]
Given $t, t_0\in\mathbb{R}$ and set $A(t)=\int_{t_0}^t \widetilde{A}(s)ds+\sum\limits_{i\in\mathscr{I}_{t_0}^{t}}B_i H_{t_i}(t)$.
We define the Kurzweil integrable map $\mathcal{Q}:\mathscr{X}\times \mathbb{R}\to \mathscr{X}$
\begin{equation*}
  \mathcal{Q}(x(t),t):=\int_{t_0}^t f(s,x(s))ds.
\end{equation*}
Then we obtain that  for every $t\geq t_0$,
\[
x(t)=x_0+\int_{t_0}^t d[A(s)]x(s)+\int_{t_0}^t D\mathcal{Q}(x(s),s),
\]
which is a solution of the following nonlinear GODEs with the initial value $x(t_0)=x_0$
\begin{equation}\label{NNGODE}
\frac{dx}{d\tau}=D[A(t)x+\mathcal{Q}(x,t)].
\end{equation}

  Now we claim that the function $\mathcal{Q}(x,t)\in \mathscr{A}(\Omega, \mu)$, here $\Omega=\mathscr{X}\times \mathbb{R}$.
In fact, for any $t,\widetilde{t}\in\mathbb{R}$ and $x\in \mathscr{X}$, by condition (c), there must exists a bounded nondecreasing function $\mu:\mathbb{R}\to\mathbb{R}$ such that
\[
\|\mathcal{Q}(x,t)-\mathcal{Q}(x,\widetilde{t})\|= \left\| \int_{\widetilde{t}}^t f(s,x)ds\right\|\leq \int_{\widetilde{t}}^t \gamma(s) ds\leq |\mu(t)-\mu(\widetilde{t})|,
\]
and for any $t,\widetilde{t}\in\mathbb{R}$ and $x, \widetilde{x} \in \mathscr{X} $, by also using condition (c), we have
\[
 \|\mathcal{Q}(x,t)-\mathcal{Q}(x,\widetilde{t})-\mathcal{Q}(\widetilde{x},t)+\mathcal{Q}(\widetilde{x},\widetilde{t})\|=
 \left\| \int_{\widetilde{t}}^t \left(f(s,x)-f(s,\widetilde{x})\right)ds\right\|
% \leq \int_{\widetilde{t}}^t  \gamma(s)  \|z-\widetilde{z}\| ds \\
 \leq \|x-\widetilde{x}\|  |\mu(t)-\mu(\widetilde{t})|.
\]

    We now show that $V_A:=\mathrm{var}_c^d A<\infty$ for all $c,d\in\mathbb{R}$ and $c<d$.
    Indeed, let $D=\{t_0,t_1,\cdots,t_{|D|}\}$ be a division of $[c,d]$. Then
\[
\sum_{j=1}^{|D|} \|A(t_j)-A(t_{j-1})\|\leq
\sum_{j=1}^{|D|} \left\| \int_{t_{j-1}}^{t_j} \widetilde{A}(s)ds \right\|
+\sum_{j=1}^{|D|} \sum\limits_{i\in\mathscr{I}_{t_{j-1}}^{t_j}} \left\|B_i H_{t_i}(t) \right\|,
\]
by using condition (a) and (b), we deduce that
\[
\sum_{j=1}^{|D|} \left\| \int_{t_{j-1}}^{t_j} \widetilde{A}(s)ds \right\|
+\sum_{j=1}^{|D|} \sum\limits_{i\in\mathscr{I}_{t_{j-1}}^{t_j}} \left\|B_i H_{t_i}(t) \right\|
\leq \int_c^d m(s)ds+C_b <\infty,
\]
that is, $V_{A}<\infty$.
Set $|\mu(t)|\leq M_\mu$ for some sufficiently small $M_\mu>0$,
and by using conditions (a), (b) and (c), we can ensure that
\[2  M_\mu (1+K(1+2K))C_b^3e^{3C_b V_{A}}V_{A}^2<1.\]
Consequently, all assumptions of Theorem \ref{Bound1} hold,  one concludes that Eq. \eqref{INL} has a unique bounded solution.
\end{proof}

\section*{Data Availability Statement}
%\hskip\parindent
%\small
No data was used for the research in this article.

\section*{Conflict of Interest}
%\hskip\parindent
%\small
The authors declare that they have no conflict of interest.

\section*{Contributions}
\hskip\parindent
\small
 We declare that all the authors have same contributions to this paper.

\end{document}